\documentclass[10pt]{amsart}

\usepackage{verbatim}

\usepackage{amscd,amsmath,amssymb}
\usepackage[mathscr]{eucal}
\usepackage{graphicx}

\theoremstyle{plain}

\numberwithin{equation}{section}

\newtheorem{theorem}{Theorem}[section]
\newtheorem{lemma}[theorem]{Lemma}
\newtheorem{proposition}[theorem]{Proposition}
\newtheorem{corollary}[theorem]{Corollary}

\newtheorem{definition}[theorem]{Definition}

\theoremstyle{remark}
\newtheorem{remark}[theorem]{Remark}
\newtheorem*{claim*}{Claim}
\newtheorem*{example*}{Example}
\newtheorem{example}[theorem]{Example}
\newtheorem*{remark*}{Remark}
\newcommand{\R}{\mathbf{R}}
\newcommand{\Rn}{\R^n}

\newcommand{\Sphere}{\mathbf{S}}
\newcommand{\Ball}{\mathbf{B}}

\DeclareMathOperator{\vol}{vol}

\newcommand{\graph}{\mathop{\rm Graph}}
\newcommand{\antigraph}{\mathop{\rm Antigraph}}
\newcommand{\dom}{\mathop{\rm Dom}}
\newcommand{\ran}{\mathop{\rm Ran}}
\newcommand{\diam}{\mathop{\rm Diam}}
\newcommand{\dist}{\mathop{\rm dist}}
\newcommand{\intr}{\mathop{\rm int}}
\newcommand{\cl}{\mathop{\rm cl}}
\newcommand{\loc}{{\mathop{\rm loc}}}
\newcommand{\Lip}{{\mathop{\rm Lip}}}
\newcommand{\normal}{{\hat n}}

\newcommand{\I}{{\tt I}}
\newcommand{\0}{{0}}

\newcommand\cexp{\hbox{$c$-Exp}}
\newcommand\Exp{\hbox{-Exp}}
\newcommand\cdexp{\hbox{$c^*$-Exp}}

\newcommand\blank{}
\newcommand\el{l}

\begin{document}

\title[Continuity, curvature, and optimal transportation]
{Continuity, curvature, and the general covariance of optimal transportation}

\author{Young-Heon Kim and Robert J. McCann
}

\address{Department of Mathematics, University of Toronto\\
  Toronto, Ontario Canada M5S 2E4}

\email{yhkim@math.toronto.edu, mccann@math.toronto.edu
}
\date{\today}

\thanks{This research was supported in part by Natural Sciences and Engineering
Research Council of Canada Grant 217006-03 and United States National Science Foundation
Grant DMS-0354729.
\\ \indent
\copyright 2007 by the authors.
}


\begin{abstract}
Let $M$ and $\bar M$ be $n$-dimensional manifolds equipped with suitable Borel
probability measures $\rho$ and $\bar \rho$.  Ma, Trudinger \& Wang gave
sufficient
conditions on a transportation cost $c \in C^4(M \times \bar M)$ to guarantee
smoothness of the optimal map pushing $\rho$ forward to $\bar \rho$;
the necessity of these conditions was deduced by Loeper. The present manuscript
shows the form of these conditions to be largely dictated by the covariance of the question;
it expresses them via non-negativity of the sectional curvature
of certain null-planes in a novel but natural pseudo-Riemannian
geometry which the cost $c$ induces on the product space $M \times \bar M$.

H\"older continuity of optimal maps was established for rougher
mass distributions by Loeper,  still relying on a key result of Trudinger \& Wang
which required certain structure on the domains and the cost.
We go on to develop this theory 
for mass distributions on
differentiable
manifolds
--- recovering Loeper's Riemannian examples such as the round sphere as particular cases ---
give a direct proof of the key result mentioned above,
and revise
Loeper's H\"older continuity argument to make it logically independent
of all earlier works, while
extending it to 
less restricted geometries and 
cost functions even for subdomains $M$ and $\bar M$ of $\Rn$.
We also give new examples of geometries satisfying the hypotheses --- obtained using
submersions and tensor products ---
and some connections to spacelike Lagrangian submanifolds in symplectic geometry.
\end{abstract}
\maketitle

\section{Introduction}

Let $M$ and $\bar M$ be Borel subsets of compact separable metric spaces,
in which their closures are denoted by $\cl M$ and $\cl \bar M$.
Suppose $M$ and $\bar M$ are equipped with
Borel probability measures $\rho$ and $\bar \rho$,
and let $c:\cl(M \times \bar M) \longrightarrow \R \cup \{+\infty\}$ be a lower semicontinuous transportation
cost defined on the product space.
The optimal transportation problem of Kantorovich \cite{Kantorovich42}
is to find the measure $\gamma \ge 0$ on $M \times \bar M$ which
achieves the infimum
\begin{equation}\label{Kantorovich}
W_c(\rho,\bar\rho)
:= \min_{\gamma \in \Gamma(\rho,\bar\rho)} \int_{M \times \bar M} c(x,\bar x) d\gamma(x,\bar x).
\end{equation}
Here $\Gamma(\rho,\bar\rho)$ denotes the set of joint probabilities having the same left and
right marginals as $\rho \otimes \bar\rho$.  It is not hard to check that this
minimum is 
attained;
any minimizing measure $\gamma \in \Gamma(\rho,\bar\rho)$ is then called {\em optimal}.
Each feasible $\gamma \in \Gamma(\rho,\bar\rho) $ can be thought of as a weighted relation
pairing points $x$ distributed like $\rho$ with points $\bar x$ distributed like $\bar\rho$;
optimality implies this pairing also minimizes the average value of the specified cost
$c(x,\bar x)$ of transporting each point $x$ to its destination $\bar x$.

The optimal transportation problem of Monge \cite{Monge81} amounts to finding a Borel map
$F:M \longrightarrow \bar M$, and an optimal measure $\gamma$ vanishing outside
$\graph(F) := \{(x,\bar x) \in M \times \bar M \mid \bar x=F(x)\}$.
When such a map $F$ exists,  it is called an {\em optimal map} between
$\rho$ and $\bar\rho$;   in this case,  the relation $\gamma$ is single-valued,
so that $\rho$-almost every point $x$ has a unique partner $\bar x=F(x)$, and
optimality can be achieved in (\ref{Kantorovich}) without subdividing the mass at such
points $x$ between different destinations. 
Although Monge's problem is more subtle to solve
than Kantorovich's,  when $M$ is a smooth manifold 
and $\rho$ vanishes on every Lipschitz submanifold of lower dimension,
a {\em twist} condition ({\bf (A1)}, Definition \ref{D:twist} below)
on the cost function $c(x,\bar x) $ guarantees existence and uniqueness of an optimal map $F$,
as well as uniqueness of the optimal measure $\gamma$; see Gangbo \cite{Gangbo95},
Gangbo \& McCann \cite{GangboMcCann96},
Carlier \cite{Carlier03}, Ma, Trudinger \& Wang \cite{MaTrudingerWang05},
or the reviews in 
\cite{ChiapporiMcCannNesheim07p} 
\cite{Villani06}.
One can then
ask about the smoothness of the optimal map $F:M \longrightarrow \bar M$.

For Euclidean distance squared $c(x,\bar x) = |x-\bar x|^2/2$, this regularity question was
resolved using geometric ideas by Caffarelli \cite{Caffarelli92} \cite{Caffarelli92b}
\cite{Caffarelli96b},  and also by Delano\"e in the plane \cite{Delanoe91} and by
Urbas in higher dimensions \cite{Urbas97}, who formulated it as an oblique boundary
value problem and applied the continuity method with a priori estimates.
Convexity of $\bar M \subset \Rn$ necessarily plays
a crucial role.  Delano\"e investigated regularity of optimal
transport with respect to Riemannian distance squared on a compact manifold \cite{McCann01},
but completed his program only for nearly flat manifolds \cite{Delanoe04}, an
improvement on Cordero-Erausquin's result from the torus \cite{Cordero-Erausquin99T},
though his criterion for nearness to flat depends on the measures $\rho$ and $\bar \rho$.
Under suitable conditions on $\rho$ and $\bar \rho$ and domains $M$ and $\bar M \subset \Rn$,
Ma, Trudinger \& Wang \cite{MaTrudingerWang05} \cite{TrudingerWang07p} developed estimates
and a continuity method approach to a class of cost functions $c \in C^4(M \times \bar M)$
which satisfy a mysterious structure condition comparing third and fourth derivatives.
Adopting the notation defined in the following section,  they express this condition
in the form
\begin{equation}\label{TWA3w}
\sum_{1 \le i,j \le n} (-c_{ij\bar k\bar \el} + c_{ij \bar a} c^{\bar a b} c_{\bar k \bar \el b}) c^{\bar k e} c^{\bar \el f}
p_i p_j q_{e} q_{f}
\ge C|p|^2|q|^2
\qquad {\rm if}\ p \perp q,
\end{equation}
for some constant $C \ge 0$ and each pair of orthogonal vectors $p, q \in \Rn$.
Here summation on $\bar a,b,e,f,\bar k,\bar \el$ is implicit but the sum
on $i,j$ is written explicitly for consistency with our later notation.
Loeper \cite{Loeper07p}
confirmed their structure condition to be necessary for continuity of $F$,
as well as being sufficient for its
smoothness \cite{MaTrudingerWang05} \cite{TrudingerWang07p} \cite{TrudingerWang08p}.
Loeper furthermore offered a
direct argument giving an explicit H\"older exponent for $F$,
largely avoiding the continuity method,
except that it relied on central results of Delano\"e, Loeper,
Ma, Trudinger, and Wang, to establish key technical lemmata
on compact Riemannian manifolds \cite{Delanoe04} \cite{MaTrudingerWang05}
such as the sphere~\cite{DelanoeLoeper06}, and on Euclidean domains \cite{TrudingerWang07p}.

Although the tensorial nature of condition (\ref{TWA3w}) was known
to both Trudinger and Loeper,  it was not emphasized in their
manuscripts.\footnote{Personal communications with Neil Trudinger and Gregoire Loeper.
This observation is made en passant in \cite{Trudinger06}, while
an assertion and proof of this fact has been included in the revised version
of \cite{Loeper07p},  which was communicated to Villani and recorded in \cite{Villani06}
independently of the present work.
Although Loeper was motivated to call the expression appearing in (\ref{TWA3w})
a {\em cost-sectional curvature} by his discoveries in the Riemannian setting
(Example \ref{E:Riemannian}),
it does not seem to have been noted by previous authors that the manifest covariance
of the regularity question largely dictates the form of its necessary and sufficient
answer,  by requiring this answer to
be expressible in terms of coordinate independent quantities such as
geodesics and curvature.}
Moreover, the Euclidean structure provides a deceptive identification between
vectors with covectors which may obscure the geometrical content of (\ref{TWA3w}),
by suggesting that the Euclidean orthogonality of two vectors may be relevant,
rather than orthogonality of a vector and covector pair.
On the other hand,  the Euclidean metric plays no role in the question being
addressed, namely the smoothness of maps $F:M \longrightarrow \bar M$ which are
optimal for the cost $c$.


Our present purpose is first to extend the theory of Loeper (and in principle,
of Ma, Trudinger \& Wang) to the transportation problem set on a pair of smooth manifolds
$M$ and $\bar M$, by finding a manifestly covariant expression of
Ma, Trudinger \& Wang's structure condition (\ref{TWA3w}), as the
sectional curvature non-negativity of certain null planes
in a pseudo-Riemannian metric on $M \times \bar M$
explored here for the first time; and second
to give an elementary and direct geometrical
proof of the key ingredients which Loeper requires,
Theorems \ref{T:second stab} and \ref{T:supporting c-convex} below,
logically independent of the methods and results of Delano\"e \cite{Delanoe04},
Delano\"e \& Loeper \cite{DelanoeLoeper06},
Ma, Trudinger \& Wang \cite{MaTrudingerWang05},
Trudinger \& Wang \cite{TrudingerWang07p},
or their subsequent work 
\cite{TrudingerWang08p}.
This makes Loeper's proof of H\"older continuity of optimal maps
self-contained,  including maps minimizing distance squared between
mass distributions whose Lebesgue densities satisfy bounds above for $\rho(\cdot)$
and below for $\bar\rho(\cdot)$ on the round sphere $M= \bar M = {\mathbf S}^n$.
As a byproduct of our approach,  we are able to relax various
geometric hypotheses on $M, \bar M$ and the cost $c$
required in previous works;  some of these relaxations were also obtained simultaneously
and independently by Trudinger \& Wang using a different approach \cite{TrudingerWang08p},
which we learned of while this paper was still in a preliminary form \cite{KimMcCann07}.

An important feature of our theory is in its geometric and global nature.
In combination with our results from \cite{KimFibration}, this allows us to
extend the conclusions of the key Theorems \ref{T:second stab} and \ref{T:supporting c-convex}
to new settings, such as the Riemannian distance squared on the product
$M = \bar M ={\mathbf S}^{n_1} \times \cdots \times  {\mathbf S}^{n_k} \times \R^l$
of round spheres, or Riemannian submersions thereof.
This is a genuine advantage of our work over previous
approaches \cite{TrudingerWang07p} \cite{TrudingerWang08p} \cite{Loeper07p};
see Example~\ref{E:new examples} and Remark~\ref{R:new results on products}


Since terms like $c$-convex, $c$-subdifferential,
and notations like $\partial^c u$ are 
used inconsistently
through the literature, and because we wish to recast the entire
conceptual framework into a pseudo-Riemannian setting, we often depart from
the notation and terminology developed by Ma, Trudinger, Wang and Loeper.
Instead,  we have tried to make the mathematics accessible
to a different readership, by choosing language intended to convey
the geometrical structure of the problem and its connection
to classical concepts in differential geometry not overly specialized
to optimal transportation or fully nonlinear differential equations.
This approach has the advantage of inspiring certain intuitions about the
problem which are quite distinct from those manifested in the previous approaches,
and has a structure somewhat reminiscent of symplectic or complex geometry.
Although we were initially surprised to discover that the intrinsic geometry
of optimal transportation is pseudo-Riemannian,  with hindsight we explain
why this must be the case, and make some connections to the theory of Lagrangian
submanifolds in the concluding remarks of the paper.

The outline of this paper is as follows.  In the next section we introduce
the pseudo-Riemannian framework and use it to adapt
the relevant concepts and structures from Ma, Trudinger \& Wang's
work on Euclidean domains to manifolds whose only geometric
structure arises from a cost function $c:M \times \bar M \longrightarrow \R \cup \{+\infty\}$.
Since Morse theory prevents a smooth cost from satisfying the desired hypothesis {\bf (A1)}
on a compact manifold, we deal from the outset with functions which may fail to be smooth on a
small set --- such as the cut-locus in the Riemannian setting \cite{McCann01} \cite{Loeper07p},
see Example \ref{E:Riemannian}, or the diagonal in the reflector antenna problem
\cite{GlimmOliker03} \cite{Wang04} \cite{CaffarelliGutierrezHuang},
see Example \ref{E:reflector antenna}.
This is followed by Section \ref{S:main results},
where we motivate and state the main theorem proved here: a version of Loeper's
global geometric characterization of (\ref{TWA3w}) which we call the
double-mountain above sliding-mountain maximum principle.  In the same section
we illustrate how this theorem and the pseudo-Riemannian framework shed new
light on a series of familiar examples from Ma, Trudinger, Wang and Loeper,
including those discussed above,  and new ones formed from these by quotients and tensor
products of, e.g., round spheres of different sizes, in Example \ref{E:new examples}.
Section~\ref{S:proofs} contains the proofs
which relate our definitions to theirs and establish the main theorem.
A level set argument is required to handle the more delicate case in which the
positivity in (\ref{TWA3w}) is not strict.
A last section offers some perspective on these results
and their connection to optimal transportation. We have included a series of
appendices which give a complete account of Loeper's theory of H\"older
continuity of optimal mappings,  illustrating how our main result makes
this theory self-contained,  and simplifying the argument at a few points.
In particular, we give a unified treatment of the Riemannian sphere and reflector
antenna problems,  using the fact that the mapping is continuous in the former
to deduce the fact that it avoids the cut-locus \cite{DelanoeLoeper06},
instead of the other way around.

It is a pleasure to thank Neil Trudinger, Xu-Jia Wang,
Gregoire Loeper, and Cedric Villani for many fruitful
discussions, and for providing us with preprints of their inspiring works.  Useful
comments were also provided by John Bland, Robert Bryant, Philippe Delano\"e,
Stamatis Dostoglou, Nassif Ghoussoub, Gerhard Huisken, Tom Ilmanen, Peter Michor,
Truyen Nguyen, and Micah Warren. We thank Jeff Viaclovsky for helpful references,  and
Brendan Pass for a careful reading of the manuscript. We are grateful to the Mathematical Sciences
Research Institute in
Berkeley, where parts of this paper were written,  and to Adrian Nachman, Jim Colliander, and
the 2006-07 participants Fields Analysis Working Group, for the stimulating environment which
they helped to create.

\section{Pseudo-Riemannian framework}
\label{S:pseudo-Riemannian}

Fix manifolds $M$ and $\bar M$ which, if not compact, are continuously embedded in separable
metrizable spaces where their closures $\cl M$ and $\cl \bar M$ are compact.
%
Equip $M$ and $\bar M$ with Borel probability measures $\rho$ and $\bar \rho$, and
a lower semicontinuous cost function $c:\cl(M \times \bar M) \longrightarrow \R \cup \{+\infty\}$,
and a subdomain $N \subset M \times \bar M$ of the product manifold. Visualize the relation
$N$ as a multivalued map and denote its inverse by $N^* := \{(\bar x,x) \mid (x,\bar x) \in N\}$.
We call $\bar N(x) := \{ \bar x \in \bar M \mid (x,\bar x) \in N\}$ the
set of destinations {\em visible} from $x$,  and
$N(\bar x) := \{ x \in M \mid (x,\bar x) \in N\}$ the set of sources
{\em visible} from $\bar x$.  We define the reflected cost $c^*(\bar x,x):=c(x,\bar x)$
on $\bar M \times M$.
In local coordinates $x^1,\ldots,x^n$ on $M$ and $x^{\bar 1},\ldots,x^{\bar n}$
on $\bar M$,  we use the notation such as $c_i = \partial c/\partial x^i$ and
$c_{\bar i} =\partial c/\partial x^{\bar i}$ to denote the partial derivatives
$Dc = (c_1, \ldots, c_n)$ and $\bar D c = (c_{\bar 1},\ldots, c_{\bar n})$ of the cost,
and $c_{i \bar j} = \partial^2 c/\partial x^{\bar j}\partial x^{i}$
to denote the mixed partial derivatives,  which commute with each other and form
the coefficients in the $n \times n$ matrix $\bar D D c$.
When $c_{i \bar j}$ is invertible its inverse matrix will be denoted by $c^{\bar j k}$.
The same notation is used for tensor indices,  with repeated indices being summed from
$1$ to $n$ (or $n+1$ to $2n$ in the case of barred indices), unless otherwise noted.

Let $T_x M$ and $T_x^* M$ denote the tangent and cotangent spaces to $M$ at $x$.
Since the manifold $N \subset M \times \bar M$ has a product structure,  its tangent
and cotangent spaces split canonically: $T_{(x,\bar x)} N = T_x M \oplus T_{\bar x} \bar M$
and $T^*_{(x,\bar x)} N = T^*_x M \oplus T^*_{\bar x} \bar M$.  For $c(x,\bar x)$ sufficiently smooth,
this canonical splitting of the one-form $dc$ will be denoted by
$dc =  D c \oplus \bar D c$.
Similarly,  although
the Hessian of $c$ is not uniquely defined until a metric has been selected on $N$,
the cross partial derivatives $\bar D D c$ at $(x,\bar x) \in N$ 
define an unambiguous linear map from vectors at $\bar x$ to covectors at $x$;
the adjoint $(\bar D D c)^\dagger = D \bar D c$ of this map takes
$T_{x} M$ to $T^*_{\bar x} \bar M$.  Thus
\begin{equation}\label{metric}
h := \frac{1}{2}\left(
\begin{array}{cc}
0 & -\bar D D c \\
-D \bar D c & 0
\end{array}
\right)
\end{equation}
gives a symmetric bilinear form on the tangent space $T_{(x,\bar x)} N$ to the
product.
Let us now adapt the assumptions of Ma, Trudinger \& Wang \cite{MaTrudingerWang05}
\cite{TrudingerWang07p} to manifolds:\\
${\bf (A0) (Smoothness)}\ c \in C^4(N).$

\begin{definition}\label{D:twist}{\bf (Twist condition)}
A cost $c \in C^1(N)$ is called {\em {twisted}} if \\
{\bf (A1)} for all $x \in M$ the map $\bar x \longrightarrow -  D c(x,\bar x)$ from
$\bar N(x) \subset \bar M$ to $T^*_{x} M$ is injective.\\
If $c$ is twisted on $N \subset M \times \bar M$ and $c^*(\bar x,x) = c(x,\bar x)$
is twisted on $N^*=\{(\bar x, x)\mid (x,\bar x) \in N\}$ we say $c$ is {\em bi-twisted}.
\end{definition}

\begin{definition}\label{D:non-degenerate}{\bf (Non-degeneracy)}
A cost $c \in C^2(N)$ is {\em non-degenerate} if \\
{\bf (A2)} for all $(x,\bar x)\in N$ the linear map $\bar D D c : T_{\bar x} \bar M \longrightarrow T^*_{x} M$
is bijective.
\end{definition}

Though {\bf (A1)} will not be needed until the appendices (and must extend
to $\cl N$ there),
for suitable probability measures $\rho$ and $\bar \rho$ on $M$ and $\bar M$  the
twist condition  alone is enough to guarantee the Kantorovich
infimum (\ref{Kantorovich}) is uniquely attained,  as well as existence of an
optimal map $F:M \longrightarrow \bar M$
\cite{Gangbo95} \cite{Carlier03} \cite{MaTrudingerWang05},
as reviewed in \cite{Villani06} \cite{ChiapporiMcCannNesheim07p}.
It implies the dimension of $\bar M$ cannot exceed that of $M$,  while {\bf (A2)}
forces these two dimensions to coincide.
The non-degeneracy condition {\bf (A2)}
ensures the map $\bar x \longrightarrow -  D c(x,\bar x)$
acts as a local diffeomorphism from $\bar N(x) \subset \bar M$ to a subset of $T_x^* M$
(which becomes global if the cost is twisted, in which case its inverse is
 called the \emph{cost-exponential} \cite{Loeper07p}, Definition \ref{c-exp} below),
and that $h(\cdot,\cdot)$ defined
by (\ref{metric}) is a non-degenerate symmetric
bilinear form on $T_{(x,\bar x)} N$. Although $h$ is not positive-definite,  it defines a
pseudo-Riemannian metric on $N$, which might also be denoted by
$d\ell^2 = - c_{i \bar j} dx^i dx^{\bar j}$.  The signature of this metric is zero, 
since in any choice of coordinates on $M$ and $\bar M$,
the eigenvalues of $h$ come in $\pm \lambda$ pairs due to
the structure (\ref{metric}); the corresponding eigenvectors are
$p \oplus \bar p$ and $(-p) \oplus \bar p$ in $T_{(x,\bar x)}N = T_x M \oplus T_{\bar x}\bar M$.
Non-degeneracy ensures there are no zero eigenvalues.  A vector 
is called {\em null} if $h(p\oplus \bar p,p \oplus \bar p)=0$.
A submanifold $\Sigma \subset N$ is called {\em null} if all its tangent vectors
are null vectors, and {\em totally geodesic} if each geodesic curve tangent
to $\Sigma$ at a point is contained in $\Sigma$ locally.  The submanifolds
$\{x\} \times \bar N(x)$ and $N(\bar x) \times \{\bar x\}$
are examples of null submanifolds in this geometry, and will turn out to be totally geodesic.
Assuming $c \in C^4(N)$,  we can use the Riemann curvature tensor ${R_{i'j'k' \el'}}$
induced by $h$ on $N$ to define the sectional curvature of a two-plane $P \wedge Q$
at $(x,\bar x) \in N$:
\begin{equation}\label{sectional curvature}
\sec_{(x,\bar x)} P \wedge Q = \sec_{(x,\bar x)}^{(N,h)} P \wedge Q
= 
\sum_{i'=1}^{2n} \sum_{j'=1}^{2n} \sum_{k'=1}^{2n} \sum_{\el'=1}^{2n}
R_{i'j'k' \el'} P^{i'} Q^{j'} P^{k'} Q^{\el'}.
\end{equation}
In this geometrical framework,  we reformulate the mysterious structure condition
(\ref{TWA3w}) of Ma, Trudinger \& Wang \cite{MaTrudingerWang05} \cite{TrudingerWang07p}
from the Euclidean setting,
which was necessary for continuity of optimal maps \cite{Loeper07p}
and sufficient for regularity \cite{MaTrudingerWang05} \cite{TrudingerWang07p}.
The reader is able to recover their condition
from ours by computing the Riemann curvature tensor (\ref{Riemann curvature}).
Note that we do not normalize our sectional curvature definition
(\ref{sectional curvature}) by dividing by the customary quantity $h(P,P) h(Q,Q) - h(P,Q)^2$,
since this quantity vanishes in the case of most interest to us,
namely $P = p \oplus \0$ orthogonal to $Q = \0 \oplus \bar p$, which means
$p \oplus \bar p$ is null.

\begin{definition}{\bf (Regular costs and cross-curvature)}
A cost $c \in C^4(N)$ is {\em weakly regular}
on $N$ if it is non-degenerate and for every $(x,\bar x) \in N$,\\
{\bf (A3w)} $\sec_{(x,\bar x)} (p \oplus \0) \wedge (\0 \oplus \bar p) \ge 0$
for all null vectors $p \oplus \bar p \in T_{(x,\bar x)} N$.\\
A weakly regular cost function is \emph{strictly regular} on $N$
if equality in {\bf (A3w)} implies $p =\0$ or $\bar p =\0$, in which case
we say {\bf (A3s)} holds on $N$.  We refer to the quantity appearing in
(\ref{full cross-curvature}) as the {\em cross-curvature},  and
say a weakly regular cost --- and the pseudo-metric (\ref{metric}) it induces on $N$ --- are
{\em non-negatively cross-curved} if
\begin{equation}\label{full cross-curvature}
\sec_{(x,\bar x)} (p \oplus \0) \wedge (\0 \oplus \bar p) \ge 0
\end{equation}
holds for all $(x,\bar x) \in N$ and $p\oplus \bar p \in T_{(x,\bar x)} N$, not necessarily null.
The cost $c$ and geometry $(N,h)$ are said to be
\emph{positively cross-curved} if, in addition, equality in
(\ref{full cross-curvature}) implies $p=\0$ or $\bar p =\0$.
\end{definition}

If $\cl M \subset\subset M'$ and $\cl \bar M \subset\subset {\bar M'}$ are
contained in larger manifolds
and {\bf (A0), (A2)} and {\bf (A3s/w)} all hold on some neighbourhood
$N' \subset M' \times {\bar M'}$ containing $N \subset \subset N'$
compactly, we say $c$ is strictly/weakly regular on $\cl N$. 
If, in addition {\bf (A1)} holds on $N'$,
we say $c$ is twisted on $\cl N$. 

The nullity condition on $p \oplus \bar p$ distinguishes weak regularity of the cost
from non-negative cross-curvature: this distinction is important in
Examples \ref{E:reflector antenna} and \ref{E:new examples}
among others; see also Trudinger \& Wang \cite{TrudingerWang07p}.
Non-negative cross-curvature is in turn a weaker condition
than $\sec^{(N,h)} \ge 0$, which means
$\sec_{(x,\bar x)} (p \oplus \bar q) \wedge (q \oplus \bar p) \ge 0$ for all
$(x,\bar x) \in N$ and $p \oplus \bar q, q\oplus \bar p \in T_{(x,\bar x)}N$.
As a consequence of Lemma \ref{L:curvature tensor},
and due to the special form of the pseudo-metric,
$\sec^{(N,h)} \ge 0$ is equivalent to requiring
non-negativity of the cross-curvature operator as a quadratic form on the vector space $T_xM \wedge T_{\bar x}\bar M$, i.e.,
\begin{equation}\label{equivalent to sectional curvature non-negativity}
R_{i\bar j k \bar \el} (p^i p^{\bar j} - q^i q^{\bar j})(p^k p^{\bar \el} - q^k q^{\bar \el})
\ge 0.
\end{equation}


\begin{example}[Strictly convex boundaries]\label{E:convex graphs}
Let $\Omega \subset \R^{n+1}$ and $\Lambda \subset \R^{n+1}$ be bounded convex
domains with $C^2$-smooth boundaries.   Set $M = \partial \Omega$,
$\bar M =\partial \Lambda$, and $c(x,\bar x) = |x-\bar x|^2/2$.
We claim the pseudo-metric (\ref{metric}) is non-degenerate and that $\sec^{(N,h)} \ge 0$
on $N:= \{ (x,\bar x) \in \partial \Omega \times \partial \Lambda \mid
\normal_\Omega(x) \cdot \normal_\Lambda(\bar x) >0 \}$, where
$\normal_\Omega(x)$ denotes the outer normal to $\Omega$ at $x$.
Indeed, fixing $(x, \bar x) \in N$,  parameterize $M$ near $x$
as a graph $X \in \Rn \longrightarrow (X,f(X)) \in \partial \Omega$
over the hyperplane orthogonal to $\normal_\Omega(x) + \normal_\Lambda(\bar x)$,
and $\bar M$ near $\bar x$ by a convex graph $\bar X \in \Rn \longrightarrow (\bar X,g(\bar X))$ over the
same hyperplane.  This choice of hyperplane guarantees $|\nabla f(X)|<1$ and
$|\nabla g(\bar X)|<1$ nearby,  so in the canonical coordinates 
and inner product on $\R^n$,  {\bf (A2)--(A3w)} follow from a computation of
Ma, Trudinger \& Wang \cite{MaTrudingerWang05} which yields the cross-curvature
\begin{equation}\label{convex boundary cross-curvature}
\sec^{(N,h)}_{(x,\bar x)} (p \oplus \0) \wedge (\0 \oplus \bar p)
= (p^i f_{ik} p^k) (p^{\bar j} g_{\bar j \bar l} p^{\bar l})/(2+ 2 \nabla f \cdot \nabla g)
\ge 0.
\end{equation}
In fact, we can also deduce the stronger conclusion
$\sec^{(N,h)} \ge 0$ as in (\ref{equivalent to sectional curvature non-negativity}):
\begin{align}
&(2+ 2 \nabla f \cdot \nabla g) \;
\sec^{(N,h)}_{(x,\bar x)}
(p \oplus \bar q) \wedge (q \oplus \bar p)
\nonumber \\&= \langle p D^2 f p\rangle \langle \bar p D^2 g \bar p \rangle
             + \langle q D^2 f q\rangle \langle \bar q D^2 g \bar q \rangle
            -2 \langle p D^2 f q\rangle \langle \bar p D^2 g \bar q \rangle
\nonumber \\&\ge \left(\sqrt{\langle p D^2 f p\rangle \langle \bar p D^2 g \bar p \rangle}
           - \sqrt{\langle q D^2 f q\rangle \langle \bar q D^2 g \bar q \rangle} \right)^2.
\label{CS complete square}
\end{align}
Noting $\normal_\Omega((X,f(X))) = (\nabla f(X),-1)$, Ma, Trudinger \& Wang's
computation shows nondegeneracy {\bf (A2)} fails at the boundary of $N$
where $\normal_\Omega(x) \cdot \normal_\Lambda(\bar x) =0$ implies
the denominator of (\ref{convex boundary cross-curvature}) is zero.
Gangbo \& McCann \cite{GangboMcCann00} showed the cost is twisted on $N$ 
provided $\Lambda$ is strictly convex, but
cannot be twisted on any larger domain in
$M \times \bar M$.
If both $\Omega$ and $\Lambda$ are $2$-uniformly convex,
meaning that the Hessians $D^2 f$ and $D^2 g$ are positive definite,
the conditions for equality in (\ref{CS complete square})
show the sectional curvature of $(N,h)$ to be positive.
  The resulting strict regularity {\bf (A3s)}
underlies Gangbo \& McCann's proof of continuity for each of the multiple mappings which
--- due to the absence of twisting {\bf (A1)} --- are required to
support the unique optimizer $\gamma \in \Gamma(\rho,\bar \rho)$ in this geometry.
Here the probability measures $\rho$ and $\bar \rho$
are assumed mutually absolutely continuous with respect to surface measure on $\Omega$ and $\Lambda$,
both having densities bounded away from zero and infinity.  For contrast,  observe
in this case that the same computations show that although nondegeneracy {\bf (A2)} also
holds on the set $N_-:= \{ (x,\bar x) \in \partial \Omega \times \partial \Lambda \mid
\normal_\Omega(x) \cdot \normal_\Lambda(\bar x) <0 \}$, this time both surfaces can
be expressed locally as graphs over $\normal_\Omega(x) - \normal_\Lambda(\bar x)$
but {\bf (A3w)} fails at each point $(x,\bar x) \in N_-$: indeed, the cross-curvatures
of $N_-$ are all negative because $D^2 f >0 > D^2 g$ have opposite signs.
\end{example}

Let us now exploit the geodesic structure which the pseudo-metric $h$ induces on
$N \subset M \times \bar M$ to recover Ma, Trudinger \& Wang's notions concerning
$c$-convex domains \cite{MaTrudingerWang05} in our setting.

\begin{definition}\label{D:c-convexity}
{\bf (Notions of convexity)}
A subset $W \subset N \subset M \times \bar M$ is {\em geodesically convex} if each
pair of points in $W$ is linked by a curve in $W$ satisfying the geodesic equation on $(N,h)$.
This definition is extended to subsets $W \subset \cl N$ by allowing geodesics
in $N$ which have endpoints on $\partial N$.
We say $\bar \Omega \subset \cl \bar M$ {\em appears convex} from $x \in M$
if $\{x\} \times \bar \Omega$ is geodesically convex 
and $\bar \Omega \subset \cl \bar N(x)$.
We say $W \subset M \times \bar M$ is {\em vertically convex} if
$\bar W(x) := \{\bar x \in \bar M \mid (x,\bar x) \in W \}$ appears convex from $x$ for each
$x \in M$.  
We say $\Omega \subset \cl M$ {\em appears convex} from $\bar x \in \bar M$
if $\Omega \times \{\bar x\}$ is geodesically convex 
and $\Omega \subset \cl N(\bar x)$.
We say $W \subset M \times \bar M$ is {\em horizontally convex} if
$W(\bar x) := \{x\in M \mid (x,\bar x) \in W\}$ appears
convex from $\bar x$ for each $\bar x \in \bar M$. If $W$ is both vertically and horizontally convex,
we say it is {\em bi-convex}.
\end{definition}

For a non-degenerate twisted cost {\bf (A0)--(A2)},
Lemma \ref{L:c-segments are geodesics} shows $\bar \Omega \subset \bar N(x)$
appears convex from $x$ if and only if $Dc(x,\bar \Omega)$ is convex in $T^*_x M$;
similarly for a bi-twisted cost
$\Omega \subset N(\bar x)$ appears convex from $\bar x$ if and only if
$\bar Dc(\Omega,\bar x)$ is convex in $T^*_{\bar x} \bar M$.
This leads immediately to notions of apparent {\em strict} convexity, and
apparent {\em uniform} convexity for such sets, and shows our definition of {\em apparent}
convexity is simply an adaptation to manifolds of the
$c$-convexity and $c^*$-convexity of Ma, Trudinger \& Wang \cite{MaTrudingerWang05}:
$\Omega$ is $c^*$-convex in their language with respect
to $\bar x$ if it appears convex from $\bar x$;
$\bar \Omega$ is $c$-convex with respect $x$ if it appears convex
from $x$;  and $\Omega$ and $\bar \Omega$ are $c^*$- and $c$-convex with respect to
each other if $N=\Omega \times \bar\Omega \subset \R^{2n}$ is bi-convex,
meaning  $Dc(x,\bar N(x)) \subset T^*_x M$ and $\bar Dc(N(\bar x),\bar x) \subset T^*_{\bar x}\bar M$
are convex domains for each $(x,\bar x) \in N$.

\begin{remark}\label{R:one dimension}
In dimension $n=1$,  strict regularity {\bf (A3s)} follows vacuously
from non-degeneracy {\bf (A2)}, since $p \oplus \bar p$ null implies
$p=\0$ or $\bar p=\0$.  For this reason we generally discuss $n \ge 2$ hereafter.
Note however, for $n=1$ and $c \in C^4(N)$ non-degenerately twisted,
$N$ is bi-convex if and only if $N(\bar x)$ and $\bar N(x)$ are homeomorphic
to intervals.  In local coordinates $x^1$ and $\bar x^1$ on $N$,
non-degeneracy implies $c_{1\bar 1} = \mp e^{\pm \lambda(x^1,\bar x^1)}$.
Comparing
$
-c_{11\bar 1\bar 1} + c_{11\bar 1} c^{\bar 1 1} c_{1 \bar 1 \bar 1} = \lambda_{1 \bar 1} |c_{1 \bar 1}|
$
with (\ref{Riemann curvature}) shows
\begin{equation}\label{1d costs}
c(x^1,\bar x^1)= \mp \int_{x_0}^{x^1}\int_{\bar x_0}^{\bar x^1} e^{\pm \lambda(s,t)} ds dt
\end{equation}
induces a pseudo-metric $h$ on $N$ for which
$
\sec_{(x,\bar x)} (p \oplus \bar 0) \wedge (\0 \oplus \bar p)
$
has the same sign as $\partial^2 \lambda/\partial x^1 \partial \bar x^1$
whenever $p \ne \0 \ne \bar p$.  If $(N,h)$ is connected its cross-curvature
will have therefore a definite sign if $\lambda(x,\bar x)$ is non-degenerate, and a
semidefinite sign if $\lambda(x,\bar x)$ is twisted.  Moreover, the sign of the cross-curvature,
sectional curvature,  and curvature operator all coincide on a product of one-dimensional
manifolds, although this would not necessarily be true on a product of surfaces or
higher-dimensional manifolds.
\end{remark}


\begin{definition}{\bf ($c$-contact set)}
Given $\Omega \subset \cl M$, $u:\Omega \longrightarrow \R \cup \{+\infty\}$, and
$c:\cl(M \times \bar M) \longrightarrow \R \cup \{+\infty\}$,
we define  $\dom c := \{ (x,\bar x) \in \cl(M \times \bar M) \mid c(x,\bar x)<\infty \}$,
the $c$-contact set $\blank\partial_\Omega^c u (x)
:= \{ \bar x \in \cl \bar M \mid (x,\bar x) \in \partial_\Omega^{c} u\}$,
and $\blank\partial^c u= \blank\partial^c_{\cl M} u$,
where
\begin{equation}\label{c-subdifferential}
\blank\partial_\Omega^{c}u := \{(x,\bar x) \in \dom c \mid
u(y) + c(y,\bar x) \ge u(x) + c(x,\bar x) \ {\rm for\ all}\ y \in \Omega\}.
\end{equation}
We define $\dom \partial^c_\Omega u := \{ x \in \Omega \mid \partial^c_\Omega u(x) \ne \emptyset \}$
and $\dom \partial^c u := \dom \partial^c_{\cl M} u$.
\end{definition}

\section{Main Results and Examples}
\label{S:main results}

A basic result of Loeper \cite{Loeper07p}
states that a cost satisfying {\bf (A0)-(A2)} on a bi-convex domain
$N = M \times \bar M \subset \R^{n} \times \R^n$ is
weakly regular {\bf (A3w)} if and only if
$\blank\partial^c u(x)$ appears convex from $x$ for each function
$u:M \longrightarrow \R \cup \{+\infty\}$ and each $x \in M$.
His necessity argument is elementary and direct,  but for sufficiency he appeals
to a result of Trudinger \& Wang which required $c$-boundedness
of the domains $M$ and $\bar M$ in the original version of \cite{TrudingerWang07p}.
The same authors gave another proof of sufficiency for strictly regular costs
in \cite{TrudingerWang08p},  and removed the $c$-boundedness
restriction in the subsequent revision of \cite{TrudingerWang07p}.
Our main result is a direct proof of this sufficiency,
found independently but simultaneously with \cite{TrudingerWang08p},
under even weaker conditions on the cost function and domain geometry.
In particular,  the manifolds $M$ and $\bar M$ in Theorem \ref{T:second stab}
need not be equipped with global coordinate charts or Riemannian metrics,
the open set $N \subset M \times \bar M$ need not have a product structure,
and the weakly regular cost need neither be twisted nor strictly regular.
This freedom proves useful in Examples \ref{E:convex graphs}
and 
\ref{E:new examples} and Remark~\ref{R:new results on products}.

\begin{theorem}\label{T:second stab}
{\bf (Weak regularity connects $c$-contact sets)}
Use a cost $c:\cl(M \times \bar M) \longrightarrow \R \cup \{+\infty\}$
with non-degenerate restriction $c \in C^4(N)$ to define a pseudo-metric
(\ref{metric}) on a horizontally convex domain $N \subset M \times \bar M$.
Fix $\Omega \subset \cl M$, $x \in M$, and a set $\bar \Omega \subset \cl\bar N(x)$
which appears convex from $x$.
Suppose $\cap_{0 \le t \le 1} N(\bar x(t))$ is 
dense in $\Omega$
for each geodesic
$t \in \mathopen[0,1\mathclose] \longrightarrow (x,\bar x(t)) \in \{x\} \times \bar \Omega$,
and $c:\cl (N) \longrightarrow \R \cup \{+\infty\}$ is continuous.
If $c$ is weakly regular {\bf (A3w)} on $N$, then
$\bar \Omega \cap \blank\partial_{\Omega}^c u(x)$ is connected
(and in fact appears convex from $x$) for each 
$u:\Omega \longrightarrow \R \cup \{+\infty\}$ with $x \in \dom u$.

\end{theorem}

To motivate the proof of this theorem and its relevance to the economics of transportation,
consider the optimal division  of mass $\rho(\cdot)$
between two target points $\bar y,\bar z \in \bar M$ in ratio $(1-\epsilon)/\epsilon$.
This corresponds to the minimization (\ref{Kantorovich}) with
$\bar \rho(\cdot) = (1-\epsilon) \delta_{\bar y}(\cdot) + \epsilon \delta_{\bar z}(\cdot)$.
If $c(x,\bar x)$ is the cost of transporting each commodity unit from $x$ to $\bar x$,
a price differential $\lambda$ between the fair market value of the same commodity at $\bar z$
and $\bar y$ will tend to balance demand $\epsilon$ and $1-\epsilon$ with supply
$\rho(\cdot)$, given the relative proximity of $\bar z$ and $\bar y$ to
the producers $\rho(\cdot)$ distributed throughout $M$. 
Here proximity is measured by transportation cost.  Since each producer
will sell his commodity at $\bar z$ or $\bar y$,  depending on which of these two
options maximizes his profit,  the economic equilibrium and optimal solution will
be given, e.g.\ as in \cite{GangboMcCann96},
by finding the largest $\lambda \in \R$ such that
\begin{eqnarray}\label{double mountain}
u(x) &=& \max \{\lambda- c(x,\bar y), -c(x,\bar z) \}
\\ {\rm yields} \qquad \epsilon &\le& \rho \left[\{x \in M \mid u(x) = - c(x,\bar z)  \}\right].
\nonumber
\end{eqnarray}
Producers in the region $\{x \in M \mid \lambda -c(x,\bar y)> - c(x,\bar z)\}$ will
choose to sell their commodities at $\bar y$, while producers in the region where the
opposite inequality holds will choose to sell their commodities at $\bar z$;
points $x_0\in M$ on the hypersurface $c(x_0,\bar y) - c(x_0,\bar z) = \lambda$ of equality
are indifferent between the two possible sale destinations $\bar y$ and $\bar z$.
We call this hypersurface the {\em valley of indifference}, since it corresponds
to a crease in the graph of the function $u$.
Loeper's observation is that for optimal mappings to be continuous,  each
point $x_0$ in the valley of indifference between $\bar y$ and $\bar z$ must also be indifferent
to a continuous path of points $\bar x(t)$ linking $\bar y=\bar x(0)$ to $\bar z=\bar x(1)$;
otherwise, he constructs a measure $\rho$ concentrated near $x_0$
for which the optimal map to a mollified version of $\bar \rho$ exhibits a discontinuous
jump,  since arbitrarily close producers will choose to supply very different consumers.
Indifference means one can choose $\lambda(t)$ such that
\begin{equation}\label{sliding mountain}
u(x) \ge \max_{0<t<1} \lambda(t) - c(x,\bar x(t))
\end{equation}
for all $x \in M$, with equality at $x_0$ for each $t \in [0,1]$.
When the path connecting $\bar y$ to $\bar z$ exists,
this equality forces $\lambda(t) = c(x_0,\bar x(t)) - c(x_0,\bar z)$;
it also forces the path $t \in [0,1] \longrightarrow (x_0,\bar x(t))$ to be a geodesic
for the pseudo-metric (\ref{metric}) on $(N,h)$, so the path $\{\bar x(t) \mid 0 \le t \le 1 \}$
appears convex from $x_0$.

We think of a function of the form $x \longrightarrow \lambda -c(x,\bar y)$ as defining the
elevation of a mountain on $M$,  {\em focused at} (or {indexed by})
$\bar y \in \bar M$. The function $u(x)$ of (\ref{double mountain})
may be viewed as a {\em double mountain}, while the maximum (\ref{sliding mountain}) may be viewed
as the upper envelope of a one-parameter family of mountains which slide as their foci
$\bar x(t)$ move from $\bar y$ to $\bar z$.
The proof of the preceding theorem
relies on the fact that the sliding mountain stays
beneath the double mountain (while remaining tangent to it at $x_0$),
if the cost is weakly regular.
In the applications below, we
take $\Omega = \cl M$ and $\bar \Omega = \cl \bar N(x)=\cl \bar M$ tacitly.


\begin{proof}[Proof of Theorem \ref{T:second stab}]
Let $c\in C^4(N)$ be weakly regular on some
horizontally convex domain $N \subset M \times \bar M$.
Fix $u:\Omega \longrightarrow \R \cap \{+\infty\}$ and $x \in \dom u$
with $\bar y, \bar z \in \bar \Omega \cap \blank\partial_\Omega^c u(x)$.
This means $u(y) \ge u(x) - c(y,\bar z) + c(x,\bar z)$ for all $y \in \Omega$,
the right hand side takes an unambiguous value in $\R \cup \{-\infty\}$,
and the same inequality holds with $\bar y$ in place of $\bar z$.
Apparent convexity of $\bar \Omega$ from $x$ implies there exists a geodesic
$t \in \mathopen]0,1\mathclose[ \longrightarrow (x,\bar x(t))$ in $(N,h)$
with $\bar x(t) \in \bar \Omega$
which extends continuously to $\bar x(0)=\bar y$ and $\bar x(1) = \bar z$.
The desired connectivity can be
established by proving $\bar x(t) \in \blank\partial_\Omega^c u(x)$ for each
$t \in \mathopen]0,1\mathclose[$,  since this means
$\bar \Omega \cap \blank\partial_\Omega^c u(x)$ appears
convex from $x$.

Horizontal convexity implies $N(\bar x(t))$ appears convex from $\bar x(t)$ for each $t\in[0,1]$,
so agrees with the illuminated set $V(x,\bar x(t)) = N(\bar x(t))$ of Definition \ref{D:visible set}.
For any $y \in  \cap_{0 \le t \le 1} N(\bar x(t))$,
the sliding mountain lies below the double mountain, i.e.,
$f(t,y) := -c(y,\bar x(t)) + c(x,\bar x(t)) \le \max \{f(0^+,y),f(1^-,y)\}$,
according to Theorem \ref{T:maximum principle} and Remark \ref{R:Lambda is open}.
Note that $f:[0,1] \times \Omega \longrightarrow \R \cup \{-\infty\}$ is a continuous
function, since $t \in [0,1] \to (x,\bar x(t)) \in \cl N$ and
$c:\cl(N)\longrightarrow \R \cup \{+\infty\}$ are continuous and their composition
is real-valued.
We therefore replace $0^+$ by $0$ and $1^-$ by $1$ and extend the inequality to
all $y \in \Omega$ using the density of $\cap_{0 \le t \le 1} N(\bar x(t))$.
On the other hand,  $\bar x(t) \in \partial_\Omega^c u(x)$ holds if and only
if $u(y) \ge u(x) + f(t,y)$ for each $y \in \Omega$ and $t \in [0,1]$.
Since $u(y) \ge u(x) + \max\{f(0,y),f(1,y)\}$ by hypothesis,  we have established
apparent convexity of $\bar \Omega \cap \partial^c_\Omega u(x)$ from $x$ and
the proof is complete.
\end{proof}

\begin{remark}[H\"older continuity]
Since connectedness and apparent convexity
survive closure, we may replace $\bar \Omega \cap \partial_\Omega^c u(x)$
by its closure (often $\partial^c u(x)$) without spoiling the result.
The apparent convexity of $\partial^c u(x)$ from $x$ hints at
a kind of monotonicity for the correspondence
$x \in M \longrightarrow \partial^c u(x)$.
A strict form of this monotonicity can be established when the cost is strictly
regular {(\bf A3s)} (Proposition \ref{P:A3S and support}),
and was exploited by Loeper to prove H\"older continuity
$F \in C^{1/(4n-1)}(M; \cl \bar M)$ of the optimal map between
densities $\rho \in L^\infty(M) $ and $\bar \rho$ with
$(1/\bar \rho) \in L^\infty(M)$ for costs which are strictly regular and
bi-twisted on the closure of a bi-convex domain
$M \times \bar M \subset \subset \R^{2n}$.
Details of his argument and conclusions are given in the appendices below.
\end{remark}


\begin{remark}[On the relevance of twist and apparent convexity to the converse]
In the absence of the twist condition,  we have
defined apparent convexity by the existence of a geodesic, which need not be extremal
or unique.  When {\bf (A3w)} fails in this general setting,
Loeper's converse argument
shows the existence of a geodesic segment
with endpoints in $(\{x\} \times \bar N(x)) \cap \partial^c u$
but which departs from this set at some points in between.
Since the twist condition and apparent convexity
imply the existence and uniqueness of geodesics
linking points in $\{x\} \times \bar N(x)$, for a twisted cost the {\em if } statement in
Theorem \ref{T:second stab} becomes necessary as well as sufficient,
a possibility which was partly anticipated in
Ma, Trudinger \& Wang \cite{MaTrudingerWang05}.
\end{remark}

\begin{remark}[Product Domains]
If $N = M \times \bar M$ the hypotheses and conclusions become simpler to state because
$N(\bar x)=M$ and $\bar N(x) = \bar M$ for each $(x,\bar x) \in N$.
If, in addition
the product
$N = M \times \bar M \subset \R^{2n}$ is a bounded Euclidean domain, we recover
the result proved by Loeper \cite{Loeper07p} based on the regularity results of Trudinger
\& Wang \cite{TrudingerWang07p}, whose hypotheses were relaxed after \cite{TrudingerWang08p}.
\end{remark}

\begin{example}[The reflector antenna and conformal geometry]\label{E:reflector antenna}
The restriction of the cost function $c(x,\bar x) = - \log |x-\bar x|$
from $\R^{n} \times \R^{n}$ to the unit sphere
$M = \bar M= \Sphere^{n-1} := \partial \Ball^{n}_1(\0)$ arises
in conjunction \cite{Wang04} \cite{GlimmOliker03} with the reflector antenna problem
studied by 
Caffarelli, Glimm, Guan, Gutierrez, Huang, Kochengin, Marder, Newman, Oliker,
Waltmann, Wang, and Wescott, among others.
It induces a metric $h$ known to satisfy {\bf (A0)-(A3s)} with
$N = (M \times \bar M) \setminus \Delta$.  Note however,
that $c(x,\bar x)$ actually defines a pseudo-metric $h^{(\infty,\infty)}$ on the larger space
$\Rn \times \Rn \setminus \Delta$ which is almost but not quite bi-convex.
Here $\Delta := \{(y,y) \mid y \in \Rn \cup \{\infty \} \}$ denotes the diagonal.
For fixed
$a \ne \bar a \in \R^{n}$,
$$\tilde c^{(a,\bar a)}(x,\bar x)
:= - \frac{1}{2}\log \frac{|x - \bar x|^2|a-\bar a|^2}{|x - \bar a|^2|a - \bar x|^2}
$$
induces a pseudo-metric $h^{(a,\bar a)}$ which coincides with $h^{(\infty,\infty)}$
on the set where both can be defined.
Moreover, $\tilde c^{(a,\bar a)}$ extends smoothly to
$(\tilde M_a \times \tilde M_{\bar a}) \setminus \Delta$
where $\tilde M = \Rn \cup \{\infty\}$ is the Riemann sphere
and $\tilde M_a = \tilde M \setminus \{a\}$.
Furthermore $h^{(a,\bar a)}$ is independent of $(a,\bar a)$,  so has a (unique)
extension $\tilde h$ to $\tilde N := (\tilde M \times \tilde M) \setminus \Delta$
which turns out to satisfy {\bf (A2)-(A3s)} on the bi-convex set $\tilde N$;
as in \cite{Loeper07p} \cite{MaTrudingerWang05} \cite{TrudingerWang07p},
this can be verified using the alternate characterization
of {\bf (A3w/s)} via concavity (/ $2$-uniform concavity) of the restriction of the function
\begin{equation}\label{A3 in TM}
q^* \in T^*_x \tilde M \longrightarrow p^i p^j c_{ij}(x,\cexp_x q^*)
\end{equation}
to the nullspace of $p \in T_x \tilde M$ in $\dom \cexp_x$; here the $\cexp$
map is defined at \eqref{dom c-exp}.  The nullspace condition is crucial,  since
the value of this function is given by
\begin{equation}\label{A3 in log}
p^i p^j f_{ij} |_{(Df)^{-1}(-q^*)} = 2 (q^*_i p^i)^2 -|p|^2|q^*|^2,
\end{equation}
where the left-hand expression in (\ref{A3 in log}) coincides with the right-hand
expression in (\ref{A3 in TM}) for general costs of the form $c(x,\bar x) = f(x-\bar x)$.
Homogeneity of the resulting manifold $(\tilde N, \tilde h)$ follows
from the symmetries $\tilde N = (F \times F)(\tilde N)$ and
$$
\tilde c^{(a,\bar a)}(x,\bar x) = \tilde c^{(F(a), F(\bar a))}(F(x),F(\bar x))
$$
under simultaneous translation $F(x) = x-y$ by $y \in \Rn$
or inversion $F(x) = x/|x|^2$ of both factor manifolds; note
the identity $|x -\bar x| = |x||\bar x|||x|^{-2}x - |\bar x|^{-2}\bar x|$.  This simplifies
the verification of {\bf (A2)-(A3s)}, since it means infinity plays no distinguished role.
Bi-convexity of $\tilde N$ follows
from the fact that the projection of the null geodesic
through
$(z,\bar x)$ and $(y,\bar x)$ onto $M$ is given by a portion of the circle in
$\Rn \cup \{\infty\}$ passing through $z,y$ and $\bar x$ --- the unique arc of this
circle (or line) stretching
from $z$ to $y$ which does not pass through $\bar x$.  From homogeneity it suffices to
compute this geodesic in the case $(z,\bar x)=(\infty,0)$;
we may further take $y=(1,0,\ldots,0)$ using
invariance under simultaneous rotations
$F(x) = \Lambda x$ by $\Lambda \in O(n)$ and dilations $F(x) = \lambda x$
by $\lambda>0$.  This calculation demonstrates that
$\cap_{0 \le t \le 1} \tilde N(\bar x(t))$ is the complement of a circular arc
--- hence a dense subset of $\tilde M$ --- for each geodesic
$t \in [0,1] \longrightarrow (x,\bar x(t)) \in \tilde N$.
Notice that if $x, \bar x(0), \bar x(1)$ all lie on the sphere $M=\bar M$,
then so does $\bar x(t)$ for each $t \in [0,1]$.
After verifying that $h$ coincides with the restriction
of $\tilde h$ to the codimension-2 submanifold $N$,
we infer for $\bar x \in \partial \Ball_1^{n}(\0)$ that
$N(\bar x) \times \{\bar x\}$ is a totally geodesic hypersurface in
$\tilde N(\bar x) \times \{\bar x\}$,  so
the horizontal and vertical geodesics, bi-convexity, and strict regularity of
$(N,h)$
are inherited directly from the geometry of $(\tilde N,\tilde h)$,
the strict regularity via Lemma~\ref{L:mixed fourth derivative}.
Although we lack a globally defined smooth cost on $\tilde N$,
we have one on $N$,
so the hypotheses and hence the conclusions of
Theorems \ref{T:second stab} and \ref{T:maximum principle}
are directly established in the reflector antenna problem,
whose geometry is also clarified:  the vertical geodesics are
products of points on $M$ with circles on $\bar M$,  where by circle
we mean the intersection of a two-dimensional plane with $\bar M = \partial \Ball^n_1(\0)$.
The geodesics would be the same for the negation
$c_+(x,y) = + \log |x-y|$ of this cost, which satisfies {\bf (A0)-(A2)} on the bi-convex
domain $N$, but violates {\bf (A3w)} for the same reason that $c$ satisfies
{\bf (A3s)}.
Using the Euclidean norm on our coordinates, the $c$-exponential (\ref{dom c-exp})
is given by $\bar x = c \Exp_x (q^*) := x - q^*/|q^*|^2$,  the optimal map takes the form
$F(x) = c \Exp_x Du(x)$, 
and the resulting
degenerate elliptic Monge-Amp\`ere type equation
(\ref{prescribed determinant})--(\ref{Monge-Ampere type})
on $\Rn$, expressed in coordinates, is
\begin{equation}\label{sigma-n Yamabe}
\det [u_{ij}(x) + 2 u_i(x)u_j(x) - \delta_{ij} |Du(x)|^2] =
\frac{|Du(x)|^{2n}\rho(x)}{\bar \rho(x-Du(x)/|Du(x)|^2)}.  
\end{equation}
Here $\bar \rho (\bar x) = O(|\bar x|^{-2n})$ as $\bar x \to \infty$ since $\bar \rho$
transforms like an $n$-form under coordinate changes.  Using the isometries
above to make $Du(x_0)=0$ at a point $x_0$ of interest,  a slight perturbation
of the standard Monge-Amp\`ere equation is recovered nearby.
The operator under the determinant (\ref{sigma-n Yamabe}) is proportional to the Schouten
tensor of a conformally flat metric $ds^2 = e^{-4u} \sum_{i=1}^n (dx^i)^2$,
so a similar equation occurs in Viaclovsky's $\sigma_n$-version of the Yamabe
problem \cite{Viaclovsky00a} \cite{Viaclovsky00b},
which has been studied by the many authors in
conformal geometry surveyed in \cite{Viaclovsky06p} and \cite{Trudinger06}.
\end{example}

\begin{example}[Riemannian manifolds]\label{E:Riemannian}
Consider a Riemannian manifold $(M=\bar M,g)$.  Taking the cost function to be
the square of the geodesic distance $c(x,\bar x)=d^2(x,\bar x)/2$ associated to $g$,
induces a pseudo-metric tensor (\ref{metric}) on the domain
$N$ where $c(x,\bar x)$ is smooth,
i.e. the complement of the cut locus. 
Moreover, the cost exponential (Definition \ref{c-exp})
reduces  \cite{McCann01} to the Riemannian exponential
\begin{equation}\label{Riemannian exponential}
\cexp_x p^* = \exp_x p
\end{equation}
with the metrical identification $p^* = g(p,\cdot)$
of tangent and cotangent space.
A curve $t \in [0,1] \longrightarrow x(t) \in N(\bar x)$
through $\bar x$
is a geodesic in $(M,g)$
if and only if the curve $\tau(t) = (x(t),\bar x)$ is a
(null) geodesic in $(N,h)$, according to Lemma \ref{L:c-segments are geodesics}.
On the diagonal $x = \bar x$, we compute 
$h(p\oplus \bar p, p \oplus \bar p) = g(p,\bar p)$,
meaning the pseudo-Riemannian space $(N,h)$ contains
an isometric copy of the Riemannian space $(M,g)$ 
along its diagonal $\Delta := \{(x,x) \mid x \in M\}$.
The symmetry $c(x,\bar x) = c(\bar x,x)$ shows $\Delta$ 
to be embedded in $N$ as a totally geodesic submanifold, and nullity of
$p \oplus \bar p \in T_{(x,x)}N$ reduces to orthogonality of $p$ with $\bar p$.
This perspective 
illuminates
Loeper's observation \cite{Loeper07p} that negativity of one Riemannian sectional
curvature at any point on $(M,g)$ violates weak regularity {\bf (A3w)} of the cost.
Indeed, the comparison of (\ref{Riemannian law of cosines}) with
Lemma \ref{L:mixed fourth derivative} allows us to recover the fact that along
the diagonal, cross-curvatures in $(N,h)$ are proportional to Riemannian curvatures in $(M,g)$:
\begin{equation}\label{diagonal cross-curvature}
\sec^{(N,h)}_{(x,x)} (p \oplus \0)\wedge (\0 \oplus \bar p) =
\frac{4}{3}\sec^{(M,g)}_x p \wedge \bar p.
\end{equation}
\end{example}

\begin{example}[The round sphere]\label{E:sphere}
In the case of the sphere $M=\bar M = \Sphere^n$ equipped
with the standard round metric,  the cut locus consists of pairs $(x,\bar x)$
of antipodal points $d(x,\bar x) = \diam M$.  Denote its complement
by $N= \{(x,\bar x) \mid d(x,\bar x)<\diam M\}$,  where
$c(x,\bar x)= d^2(x,\bar x)/2$ is smooth. In this case, the identification
(\ref{Riemannian exponential}) of cost exponential with Riemannian
exponential implies bi-convexity of $N$,  since the cut locus forms a circle
(hypersphere if $n >2$)
in the tangent space $T_x \Sphere^n$, and the verification of ({\bf A3s}) both on and
off the diagonal was carried out by Loeper \cite{Loeper07p}.  In fact more is true:
$(N,h)$ is non-negatively cross-curved, as we verify in \cite{KimFibration}.
Given an $h$-geodesic
$t \in [0,1] \longrightarrow (x,\bar x(t)) \in N$,  we find
$\cap_{0 \le t \le 1} N(\bar x(t))$ exhausts $\Sphere^n$ except for the
antipodal curve to the exponential image
$t \in [0,1] \longrightarrow \bar x(t)$ of a line segment in $T_x \Sphere^n$.
Although this curve does not generally lie on a great circle
or even a circle, its complement is dense in $\Sphere^n$, whence
the double-mountain above sliding-mountain property and
Theorem \ref{T:second stab} follow.
\end{example}


\begin{example}[Positive versus negative curvature]
To provide more insight into Loeper's results ---
continuity of optimal maps on the round sphere versus discontinuous optimal maps
on the saddle or hyperbolic plane ---  consider dividing a smooth positive
density $\rho$ (say uniform on some disk of volume $\sum_{i=1}^3 \epsilon_i = 1$)
optimally between three points $\bar x_1, \bar x_2, \bar x_3$ on a geodesic through its center:
$\bar \rho(\cdot) = \sum_{i=1}^3 \epsilon_i \delta_{\bar x_i}(\cdot)$.
The solution to this problem is given, e.g.\ \cite{GangboMcCann96}, by finding
constants $\lambda_1,\lambda_2,\lambda_3 \in \R$ for which the function
\begin{eqnarray}\label{triple peak}
u(x)=\max_{1 \le i \le 3} \{u_i(x)\} &{\rm given\ by}& u_i(x) :=
-\lambda_i -c(x,\bar x_i) \\ \nonumber
{\rm solves}\
\epsilon_i=\rho[\Omega_i] &{\rm with} & \Omega_i := \{x \in M \mid u(x)=u_i(x)\}
\end{eqnarray}
for $i=1,2,3$.  The regions $\Omega_i$ are illustrated in Figure 1
for the case where cost $c(x,\bar x) = d^2(x,\bar x)/2$ is proportional
to Riemannian distance squared on the (a) round sphere,
(b) Euclidean plane, (c) hyperbolic disc.

\resizebox{12cm}{!}{\includegraphics{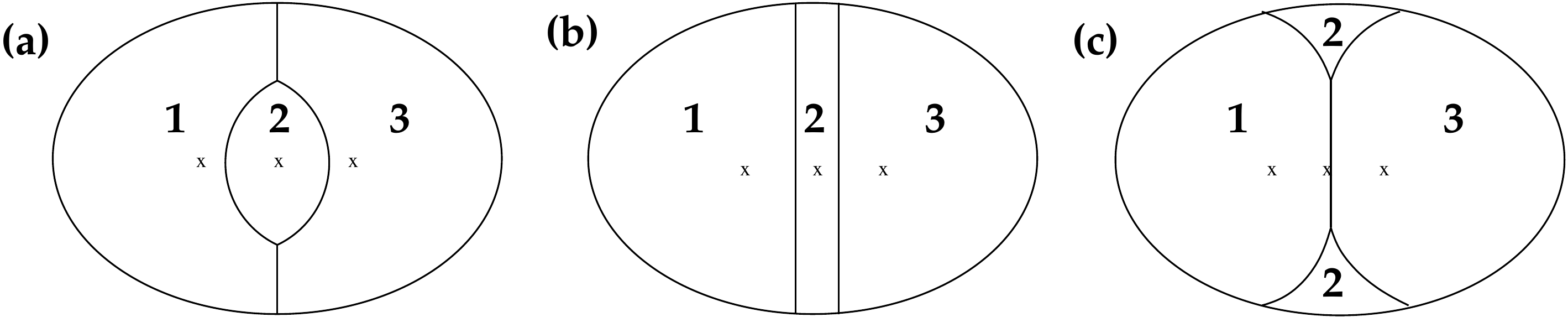}}
$$\centerline{Figure 1}$$

\noindent
Notice that only in the third
case is one of the regions $\Omega_2$ disconnected.
Now consider transporting a smoothly smeared and positive approximation $\bar \rho_\eta(\cdot)$
of the discrete source $\bar \rho(\cdot)$, obtained by mollifying on scale $\eta>0$, to the uniform
measure $\rho(\cdot)$.
For $\eta>0$ sufficiently small, points $\bar x$ near $\bar x_2$ will need to map into
a $\delta$-neighbourhood of $\Omega_2$.  It is not hard to believe the results of
Loeper's calculation for case (c),  namely that if $\eta>0$ is small the map from
$\bar \rho_\eta(\cdot)$ to $\rho(\cdot)$ will
have a discontinuity near $\bar x_2$ separating the regions which map into disjoint
$\delta$-neighbourhoods separating the two disconnected components of $\Omega_2$.
\end{example}

\begin{example}[New examples from old: perturbations, submersions, and products]
\label{E:new examples}
It remains interesting to find more general sufficient conditions
on a Riemannian manifold $(M,g)$ and function $f$ for the pseudo-metric
$h$ induced on the complement of the cut locus $N \subset M \times M$ by
$c(x,\bar x) = f(d(x,\bar x))$ to be strictly or weakly regular.
It is clear that slight perturbations of an {\bf (A3s)} cost remain
strictly regular if the perturbation is smooth and small enough.
Delano\"e \& Ge have quantified this observation on the round sphere \cite{DelanoeGe}.
It is also possible to deduce that the local properties such as
{\bf (A3s/w)} and non-negative cross-curvature
and the global property such as the conclusion of Theorem~\ref{T:second stab}
all survive Riemannian submersion \cite{KimFibration}, holding for example on quotients
of the round sphere under discrete \cite{DelanoeGe} \cite{KimFibration} or continuous
\cite{KimFibration} group actions, including in particular all spaces of constant
positive curvature \cite{DelanoeGe}, and the Fubini-Study metric on ${\bf CP}^n$
\cite{KimFibration}.
On the other hand,  an example constructed by one of us shows that
positive but non-constant sectional curvature of the underlying manifold $(M =\bar M,g)$
does not guarantee weak regularity of the cost $c=d^2/2$ away from the diagonal in
$N \subset M \times M$
\cite{KimCounterexample}.  As a final important example,
consider two manifolds
$N_+ \subset M_+ \times \bar M_+$ and $N_- \subset M_- \times \bar M_-$ equipped with
cost functions $c_\pm \in C^4(N_\pm)$ inducing
pseudo-metrics $d\ell_\pm^2 = - c_{i \bar j} d x_\pm^i dx_\pm^{\bar j}$.
As a consequence of Lemma \ref{L:mixed fourth derivative},
the product metric $d\ell^2 = d\ell_+^2 + d\ell_-^2$ corresponding
to the cost function $c_+(x_+,\bar x_+) + c_-(x_-,\bar x_-)$ on
$N = (M_+ \times M_-) \times (\bar M_+ \times \bar M_-)$
is non-negatively cross-curved (\ref{full cross-curvature}) if both $(N_\pm,h_\pm)$ are.
Although it is not true that weak regularity of the factors
implies the same for the product,
many of the known examples of weakly regular costs
(including those of Examples \ref{E:convex graphs}, \ref{E:sphere}
 and the sumbersions above) actually turn out to be non-negatively
cross-curved  \cite{KimFibration},  so this product construction
becomes a fruitful source of new examples.   Furthermore, since geodesics in the product
are products of geodesics,  bi-convexity of
the factors $(N_\pm,h_\pm)$ implies bi-convexity of the product manifold $(N,h)$.
Because a product geodesic may have constant factors,
it is not hard to show that the non-negatively curved product manifolds
$(N,h)$ always fail to be strictly regular even when both factor manifolds $(N_\pm,h_\pm)$
are positively cross-curved.  Thus tensor products of positively cross-curved costs
yield a new source of weakly regular
costs that fail to be strictly regular --- the very simplest example of
which is given by arbitrary sums $c(x,\bar x) = \sum_{k=1}^n c(x^k, x^{\bar k}; k)$ of
$k=1,\ldots,n$ positively cross-curved costs $c(s,t;k)$ as in (\ref{1d costs})
on bi-convex subdomains $N_1,\ldots, N_n \subset \R^2$. 
\end{example}

\begin{remark}\label{R:new results on products}(Products ${\mathbf S}^{n_1} \times \cdots \times {\mathbf S}^{n_k} \times \R^l$ and their Riemannian submersions)
The conclusion of Theorem~\ref{T:second stab} holds for the distance squared cost on the
Riemannian product $M = \bar M$ of round spheres
${\mathbf S}^{n_1} \times \cdots \times {\mathbf S}^{n_k}\times \R^l$ --- or its Riemannian
submersions --- by combining the preceding example with the result of \cite{KimFibration}.
The weak regularity {\bf (A3w)} of the
cost on $N = M \times \bar M \setminus \{\hbox{\rm cut locus}\}$
and the biconvexity of $N$ are  satisfied as in Example~\ref{E:new examples}.
The density condition of $\cap_{0\le t \le 1}N(\bar x(t))$ follows also easily since the cut
locus of one point in this example is a smooth submanifold of codimension greater than or
equal to $2$. This new global result illustrates an advantage of our method over
other approaches \cite{TrudingerWang07p} \cite{TrudingerWang08p} \cite{Loeper07p},
which would require a regularity result for optimal maps (or some a priori estimates)
to obtain the conclusion of Theorem~\ref{T:second stab}. To implement such
approaches for the manifolds of this example, one would need to establish that an optimal map
remains \emph{uniformly away from the cut locus}, as is currently known only
for a single sphere $M = \bar M = {\mathbf S}^n$ from work of
Delan\"oe \& Loeper \cite{DelanoeLoeper06}
(alternately \cite{Loeper07p} or Corollary~\ref{C:UAC}).
To the best of our knowledge, no one has yet succeeded in establishing regularity results
for this product example,  though one could try to obtain them from
Theorems~\ref{T:second stab} and \ref{T:supporting c-convex}, by
extending the approach we develop successfully in Appendix~\ref{S:spherical continuity}
for a single sphere $M=\bar M = {\mathbf S}^n$.
\end{remark}


\section{Proof of Main Results}
\label{S:proofs}


Let us begin by establishing coordinate representations of the Christoffel
symbols and Riemann curvature tensor for the pseudo-metric (\ref{metric}).

\begin{lemma}\label{L:curvature tensor}
{\bf (Riemann curvature tensor and Christoffel symbols)}
Use a non-degenerate cost $c \in C^4(N)$ to define a pseudo-metric
(\ref{metric}) on the domain $N \subset M \times \bar M$.
In local coordinates $x^1,\ldots,x^n$ on $M$ and
$x^{\bar 1},\ldots,x^{\bar n}$ on $\bar M$, the only non-vanishing
Christoffel symbols are
\begin{equation}
{\Gamma_{i j}}^{m} = c^{m {\bar k}} c_{{\bar k} i j} \qquad {\rm and}\qquad
{\Gamma_{{\bar i} {\bar j}}}^{{\bar m}} = c^{{\bar m} k} c_{k {\bar i} {\bar j}}.
\label{nonvanishing Christoffel symbols}
\end{equation}
Furthermore, the components of the Riemann curvature tensor (\ref{sectional curvature})
vanish except when the number of barred and unbarred indices is equal, in which
case the value of the component can be inferred from $R_{ij\bar k \bar \el}=0$ and
\begin{equation}\label{Riemann curvature}
2 R_{i {\bar j} {\bar k} \el} =
c_{i {\bar j} {\bar k} \el} - c_{\el i {\bar f}} c^{a {\bar f}} c_{a {\bar j} {\bar k}}.
\end{equation}
\end{lemma}

\begin{proof}
From
$$
{\Gamma_{i j}}^{m} := \frac{1}{2} h^{m k} (h_{k j, i} + h_{i k, j} - h_{i j, k})
+ \frac{1}{2} h^{m \bar k} (h_{\bar k j, i} + h_{i \bar k, j} - h_{i j, \bar k}),
$$
and the analogous definition with $i,j,$ and/or $m$ replaced by $\bar i,\bar j$ and $\bar m$
respectively, the off-diagonal form (\ref{metric}) of the pseudo-metric and the equality of mixed partials
implies the only non-vanishing Christoffel symbols are given by
\begin{eqnarray}
\nonumber
{\Gamma_{i j}}^{m}
&=& -\frac{1}{2} c^{m {\bar k}} (-c_{{\bar k} j i} - c_{i {\bar k} j} + 0)
\\ &=& c^{m {\bar k}} c_{{\bar k} i j}
\nonumber
\end{eqnarray}
and $\Gamma_{\bar i \bar j}^{\bar k} = c^{\bar m k} c_{k \bar i \bar j}$.
Since the only nonvanishing Christoffel
symbols are given by (\ref{nonvanishing Christoffel symbols}),  it is not
hard to compute the relevant components of Riemann's curvature tensor in the
coordinates we have chosen:
\begin{eqnarray}
\nonumber
{R_{i{\bar j}{\bar k}}}^{{\bar m}}
&=& - \frac{\partial}{\partial x^i} \Gamma_{{\bar j}{\bar k}}^{{\bar m}}
  + \frac{\partial}{\partial x^{\bar j}} \Gamma_{i{\bar k}}^{{\bar m}}
  + \Gamma_{{\bar k} i}^{{f}} \Gamma_{{\bar j} {f}}^{{\bar m}}
  - \Gamma_{{\bar k} {\bar j}}^{{f}} \Gamma_{i {f}}^{{\bar m}}
  + \Gamma_{{\bar k} i}^{{\bar f}} \Gamma_{{\bar j} {\bar f}}^{{\bar m}}
  - \Gamma_{{\bar k} {\bar j}}^{{\bar f}} \Gamma_{i {\bar f}}^{{\bar m}}
\\ &=& - \frac{\partial}{\partial x^i}(c^{{\bar m} a} c_{a {\bar j} {\bar k}})
\nonumber
\\ &=& - c^{{\bar m} b} c_{i {\bar j} {\bar k} b} + c^{{\bar m} b} c_{b i {\bar f}} c^{a {\bar f}} c_{a {\bar j} {\bar k}}.
\nonumber
\end{eqnarray}
Using the pseudo-metric (\ref{metric}) to lower indices yields (\ref{Riemann curvature}),
and the other non-vanishing components of the Riemann tensor can then be deduced
from the well-known symmetries
$-R_{\bar j i \bar k \el} = R_{i \bar j \bar k \el} = R_{\bar k \el i \bar j}
= -R_{\el \bar k i \bar j}$.  To see that the remaining components all vanish,
it suffices to repeat the analysis starting from the analogous definitions of
${R_{ij\bar k}}^{\bar m}$, ${R_{ij\bar k}}^{m}$, and ${R_{ijk}}^{\bar m}$ .
\end{proof}

\begin{remark}[Vanishing curvatures]
The vanishing of $R_{ijk\el}$, $R_{ijk \bar \el}$, $R_{i \bar j \bar k \bar \el}$,
$R_{\bar i \bar j \bar k \bar \el}$, and $R_{ij \bar k \bar \el}=0$ imply that
$(\wedge^2 T_{}M) \oplus (\wedge^2 T_{}\bar M)$
lies in the null space of curvature operator viewed as a quadratic form on
$\wedge^2 T_{} N = (\wedge^2 T_{}M) \oplus (\wedge^2 T_{}\bar M)
\oplus (T_{}M \wedge T_{}\bar M)$.
Strict / weak regularity of the cost, and the signs of the
cross-curvature (\ref{full cross-curvature}), and sectional curvature
(\ref{equivalent to sectional curvature non-negativity}) 
are therefore all
determined by the action of this operator on the $n^2$ dimensional vector bundle
$T_{}M \wedge T_{} \bar M$.  Since the cone of null vectors is nonlinear,
it is not obvious whether {\bf (A3w)} implies non-negativity of all
cross-curvatures of $(N,h)$,  but Trudinger \& Wang \cite{TrudingerWang07p}
prove, as in Example \ref{E:reflector antenna}, that this is not so. 
\end{remark}

We next recall an important map 
of Ma, Trudinger \& Wang
\cite{MaTrudingerWang05}, 
called the {\em cost-exponential} by Loeper \cite{Loeper07p}.

\begin{definition}\label{c-exp}{\bf (Cost exponential)}
If $c \in C^2(N)$ is twisted {\bf (A1)}, we define the {\em $c$-exponential} on
\begin{align}\label{dom c-exp}\nonumber
\dom(c\text{-Exp}_{x} ) :=& - D c(x, \bar N(x)) \\
=& \{ p^* \in T^*_{x} M \  | \ p^* = - D c(x, \bar x) \hbox{ \ \ for some $\bar x \in \bar N(x)$}  \}
\end{align}
by $\cexp_{x} p^* = \bar x$ if $p^*=- D c(x,\bar x)$.
Non-degeneracy {\bf (A2)} then implies the $c$-exponential is a diffeomorphism from
$\dom(\cexp_{x}) \subset T^*_{x} M$ onto $\bar N(x) \subset \bar M$.  If $c\in C^2(N)$
is non-degenerate but not twisted and $q^*= -Dc(x,\bar y)$,  the implicit function theorem
implies a single-valued branch of $\cexp_x$ 
taking values near $\bar y$ is defined by the same formula
in a small neighbourhood of $q^*$,  though it no longer extends to be a global
diffeomorphism of $\dom(\cexp_{x})$ onto $\bar N(x)$.
\end{definition}

\begin{lemma}\label{L:c-segments are geodesics}
{\bf (the $c$-segments of \cite{MaTrudingerWang05} are geodesics)}
Use a non-degenerate cost $c \in C^4(N)$ to define a pseudo-metric
(\ref{metric}) on the domain $N \subset M \times \bar M$.
Fix ${x} \in M$. If $p^*,q^* \in \dom(\cexp_{x}) \subset T_{x}^*M$ are close enough
together there will be a branch of $\cexp_{x}$ defined on the line segment
joining $p^*$ to $q^*$.  Then the curve
$s \in [0,1] \longrightarrow %
\sigma(s) := (x,\cexp_{x} ((1-s)p^* + s q^*))$
is an affinely parameterized null geodesic in (N,h).
Conversely, every geodesic segment in the totally geodesic submanifold $\{x\} \times \bar N(x)$
can be parameterized locally in this way.
\end{lemma}

\begin{proof}
Given  $s_0 \in [0,1]$,
introduce coordinates on $M$ and $\bar M$ around $\sigma(s_0)$ so that nearby, the curve
$\sigma(s)$ can be represented in the form $(x^1,\ldots,x^n,x^{\bar 1}(s),\ldots,x^{\bar n}(s))$.
Differentiating the definition of the cost exponential
\begin{equation}\label{c-segment}
0 = (1-s)p^*_i + sq^*_i + c_i(\sigma(s))
\end{equation}
twice with respect to $s$ yields
\begin{equation}\label{pre-acceleration along c-segment}
0 
= c_{i {\bar j}} \ddot x^{\bar k} + c_{i {\bar j} {\bar k} } \dot x^{\bar j} \dot x^{\bar k}
\end{equation}
for each $i=1,\ldots,n$.  Multiplying by the inverse matrix $c^{{\bar m} i}$
to $c_{i {\bar j}}$ yields
\begin{equation}\label{acceleration along c-segment}
0 = \ddot x^{\bar m} + c^{{\bar m} i} c_{i {\bar j} {\bar k}} \dot x^{\bar j} \dot x^{\bar k},
\end{equation}
always summing on repeated indices.  Together with $\ddot x^m= 0=\dot x^m$,
we claim these reduce to the geodesic equations
\begin{eqnarray}\label{general geodesic}
0 &=& \ddot x^m + {\Gamma_{i j}}^m \dot x^i \dot x^j + {\Gamma_{\bar i j}}^m \dot x^{\bar i} \dot x^j
                + {\Gamma_{i \bar j}}^m \dot x^i \dot x^{\bar j} + {\Gamma_{\bar i \bar j}}^m \dot x^{\bar i} \dot x^{\bar j} \\
0 &=& \ddot x^{\bar m} + {\Gamma_{i j}}^{\bar m} \dot x^i \dot x^j + {\Gamma_{\bar i j}}^{\bar m} \dot x^{\bar i} \dot x^j
                + {\Gamma_{i \bar j}}^{\bar m} \dot x^i \dot x^{\bar j} + {\Gamma_{\bar i \bar j}}^{\bar m} \dot x^{\bar i} \dot x^{\bar j}
\nonumber
\end{eqnarray}
on $(N,h)$.
Indeed, this follows since the only non-vanishing Christoffel symbols
are given by (\ref{nonvanishing Christoffel symbols}).
Comparing (\ref{acceleration along c-segment}) with (\ref{general geodesic}),
we see $\sigma(s)$ is an affinely parameterized geodesic near $\sigma(s_0)$;
it is null since
$h(\dot \sigma,\dot \sigma) =0$ from the off-diagonal form of (\ref{metric}).

Conversely,  any geodesic segment in $(N,h)$ which lies in $x \times \bar N(x)$ can
be parameterized affinely on $s \in [0,1]$.  Near $s_0 \in [0,1]$ it then satisfies
(\ref{pre-acceleration along c-segment}), whence
\begin{equation}\label{cotangent line}
0 = \frac{d^2}{ds^2} c_{i}(x,\bar x(s)).
\end{equation}
Integrating twice,  the constants of integration determine $p^*$ and $q^* \in T^*_x M$
such that (\ref{c-segment}) holds locally. Thus $(1-s_0)p^* + s_0 q^* \in \dom(\cexp_x)$.
Choosing a branch of the cost exponential defined near this point and
equalling $D_i c(x,\bar x(s_0))$ there,  we deduce $(x,\bar x(s)) = \cexp_x ((1-s)p^* + sq^*)$
for $s$ near $s_0$ from the definition of this branch.

Finally,  to see that $\{x\} \times \bar N(x)$ is totally geodesic,  take any point
$\bar x \in \bar N(x)$ and tangent vector $\bar p \in T_{\bar x} \bar M$.  Setting $x^m(s) = x^m$ to be constant
solves half of the geodesic equations,  since $\Gamma^{m}_{{\bar i} j} =0 = \Gamma^m_{\bar i \bar j}$.
we can still solve the remaining $n$ components of the geodesic equation
(\ref{acceleration along c-segment}) for small $s \in \R$,
subject to the initial conditions $\bar x(0) = \bar x$ and $\dot {\bar x}(0) = q$, to find a geodesic
which remains in the $n$-dimensional submanifold $\{x\} \times \bar N(x)$ for short times.
\end{proof}

The next lemma gives a non-tensorial expression of the sectional curvature
in our pseudo-Riemannian geometry $(N,h)$.  In the context of Example \ref{E:Riemannian},
it can be viewed as a generalization of the asymptotic formula for the Riemannian distance
between two arclength parameterized geodesics $x(s)$ and $\bar x(t)$ near a point
of intersection $x(0)=\bar x(0)$ at angle $\theta$:
\begin{equation}\label{Riemannian law of cosines}
d^2(x(s),\bar x(t)) = s^2 + t^2 - 2 st \cos\theta - \frac{k}{3} s^2 t^2 \sin^2\theta + O((s^2+t^2)^{5/2})
\end{equation}
where the Riemannian curvature $k$ of the $2$-plane $\dot x(0) \wedge \dot {\bar x}(0)$
on $(M = \bar M, g)$ gives the leading order correction to the law of cosines.
Though we do not need it here,  the proof of the next lemma can also be adapted to establish
an expansion analogous to (\ref{Riemannian law of cosines}) for general costs
$c(x(s),\bar x(t))$; the zeroth and first order terms do not vanish,  but the coefficients
of $s^2 t$ and $s^3 t$ are zero due to the geodesy of $s\in [0,1] \longrightarrow \sigma(s)$.
Remarkably however, to determine the coefficient of $s^2 t^2$ in the lemma below
requires only one (in fact, either one) and not both of the two curves to be geodesic.

\begin{lemma}\label{L:mixed fourth derivative}
{\bf (Non-tensorial expression for curvature)}
Use a non-degenerate cost $c \in C^4(N)$ to define a pseudo-metric
(\ref{metric}) on the domain $N \subset M \times \bar M$.
Let $(s,t) \in \mathopen[-1,1\mathclose]^2 \longrightarrow (x(s),\bar x(t)) \in N$
be a surface
containing two curves $\sigma(s)=(x(s), \bar x(0))$ and
$\tau(t) = (x(0),\bar x(t))$ through 
$(x(0),\bar x(0))$.  Note $\0 \oplus \dot {\bar x}(0)$ defines a parallel vector-field along
$\sigma(s)$. If $s \in [-1,1] \longrightarrow \sigma(s) \in N$ is a geodesic in $(N,h)$
then
\begin{equation}\label{mixed fourth derivative}
- 2 \frac{\partial^4}{\partial s^2 \partial t^2}\bigg|_{s=0=t}
c(x(s),\bar x(t)) = \sec_{(x(0),\bar x(0))} 
\frac {d\sigma}{ds} \wedge \frac {d \tau}{dt}.
\end{equation}
\end{lemma}

\begin{proof}
Introduce coordinates $x^1,\ldots,x^n$ in a neighbourhood of $x(0)$ on $M$ and 
$x^{\bar 1},\ldots,x^{\bar n}$ in a neighbourhood of $\bar x(0)$ on $\bar M$,
so the surface $(x(s),\bar x(t)) \in N$ has coordinates
$(x^1(s),\ldots,x^n(s),x^{\bar 1}(t),\ldots,x^{\bar n}(t))$ locally.
To see $0 \oplus \dot {\bar x}(0)$ defines a parallel vector field along $\sigma(s)$,
we use the Levi-Civita connection to compute
$$
\dot \sigma^ i \nabla_i \dot x^{\bar k} + \dot \sigma^{\bar i} \nabla_{\bar i} \dot x^{\bar k}
= \dot x^i \frac{\partial \dot x^{\bar k}}{\partial x^i}
 + \dot x^i \Gamma_{i j}^{{\bar k}} \dot x^{j}
 + \dot x^i \Gamma_{i {\bar j}}^{{\bar k}} \dot x^{\bar j} + 0
= 0
$$
since the only non-vanishing Christoffel symbols are given by
(\ref{nonvanishing Christoffel symbols}).

Computing the fourth mixed derivative yields
\begin{eqnarray}
\label{compute mixed fourth derivative} 
&&\frac{\partial^4 }{\partial s^2 \partial t^2} \bigg|_{s=0=t} c(x(s),\bar x(t))
\\=&& c_{i j {\bar k} {\bar \el}} \dot x^i \dot x^j \dot x^{\bar k} \dot x^{\bar \el}
+ c_{a {\bar k} {\bar \el}} \ddot x^a \dot x^{\bar k} \dot x^{\bar \el}
+ (c_{i j {\bar b}} \dot x^i \dot x^j
+ c_{a {\bar b}} \ddot x^a) \ddot x^{\bar b}
\nonumber \\ =&& (c_{i {\bar k} {\bar \el} j}
- c_{{\bar k} {\bar \el} a} c^{a {\bar b}} c_{{\bar b} i j}
)\dot x^i \dot x^j \dot x^{\bar k} \dot x^{\bar \el}
\nonumber
\end{eqnarray}
where the form (\ref{acceleration along c-segment}) of the geodesic
equation has been used to eliminate the coefficient of $\ddot x^{\bar b}$
and express $\ddot x^a$ in terms of $\dot x^i$.
Comparing (\ref{compute mixed fourth derivative}) with (\ref{Riemann curvature})
and (\ref{sectional curvature}) yields the desired conclusion (\ref{mixed fourth derivative}).
The minus sign comes from antisymmetry
 $R_{i \bar k \bar \el j} = -R_{i \bar k j \bar \el}$ of the Riemann tensor.
\end{proof}




Our next contribution culminates in Theorem \ref{T:maximum principle},
which generalizes
the result that Loeper  \cite{Loeper07p} deduced
from Trudinger \& Wang \cite{TrudingerWang07p}. 
As mentioned above, it can be interpreted to mean that
if a weakly regular function $c \in C^4(N)$ governs the cost of transporting
a commodity from the locations where it is produced to the locations where it is consumed,
a shipper indifferent between transporting the commodity from
$y$ to the consumer at either endpoint $\bar x(0)$ and $\bar x(1)$ of the geodesic
$t \in [0,1] \longrightarrow  (y, \bar x(t))$ in $(N,h)$,
will also be indifferent to transporting goods from $y$
to the consumers at each of the intermediate points $\bar x(t)$ along this geodesic.
As was also mentioned, for non-degenerate costs Loeper showed
this conclusion fails unless the cost is weakly regular.

\begin{proposition}{\bf (Maximum principle)}\label{P:critical implies min}
Use a weakly regular cost $c \in C^4(N)$ to define a pseudo-metric
(\ref{metric}) on the domain $N \subset M \times \bar M$.  Given $x \ne y \in M$,
let $t\in]0,1[ \longrightarrow (x, \bar x(t)) \in N$ be a geodesic in $(N,h)$ with
$\dot {\bar x}(1/2) \ne \0$ and set $f(t) = - c(y,\bar x(t)) + c(x, \bar x(t))$.
If $\dot f(t_0)=0$ for some $t_0 \in]0,1[$ with
a geodesic linking $(x,\bar x(t_0))$ to $(y,\bar x(t_0))$
lying in $N(\bar x(t_0))\times \{\bar x(t_0)\}$, then $\ddot f(t_0) \ge 0$.
Strict inequality holds if the relevant cross-curvature of $c$ is positive
at some point on the second geodesic.
\end{proposition}

\begin{proof}
Suppose $f(t)$ has a critical point at some $0<t_0<1$, so
\begin{eqnarray}\label{critical point}
0 = & \dot f(t_0)& = \left(\partial_{\bar i} c(x,\bar x(t_0)) - \partial_{\bar i} c(y,\bar x(t_0))\right) \dot x^{\bar i}(t_0);
\end{eqnarray}
we then claim $\ddot f(t_0) \ge 0$.

Let $s \in [0,1] \longrightarrow (x(s),\bar x(t_0))$ be a geodesic in
$N(\bar x(t_0)) \times \{\bar x(t_0) \}$ 
with endpoints
$x(0) = x$ and $x(1)= y$.  Setting $g(s,t) = -c(x(s),\bar x(t)) + c(x, \bar x(t))$,
Lemma \ref{L:mixed fourth derivative} yields
\begin{equation}\label{strict convexity}
\frac{\partial^4 g}{\partial s^2 \partial t^2} \bigg|_{(s, t_0)}
= \frac{1}{2}\sec_{(x(s),\bar x(t_0))} (\dot x(s) \oplus \0) \wedge (\0 \oplus \dot {\bar x}(t_0))
\ge 0,
\end{equation}
with the inequality following from weak regularity {\bf (A3w)} of the cost $c\in C^4(N)$,
as long as $\dot x(s) \oplus \0_{\bar x(t_0)}$ and $\0_{x(s)} \oplus \dot {\bar x}(t_0)$ are orthogonal vectors
on (N,h), or equivalently, as long as $\dot x(s) \oplus \dot {\bar x}(t_0)$ is null.
These vectors are non-vanishing since $x(0) \ne x(1)$ and $\dot{\bar x}(0) \ne \0$;
we now deduce their orthogonality from (\ref{critical point}), using subscripts to
distinguish which tangent space the zero vectors reside in.

Along the geodesic
$s \in [0,1] \longrightarrow (x(s),\bar x(t_0))$, the vector field
$\0_{x(s)} \oplus \dot {\bar x}(t_0)$ is parallel transported
according to Lemma \ref{L:mixed fourth derivative}.
Thus the inner product $\lambda$ of this vector field
with the tangent vector is independent of $s \in [0,1]$.
Define $q^*_{\bar i}(s) := \partial_{\bar i} c(x(s),\bar x(t_0)) \in T_{\bar x(t_0)}^* \bar M$,
so $\dot q^*_{\bar i}(s) = \partial_j \partial_{\bar i} c (x(s),\bar x(t_0)) \dot x^j(s)$. From
the form (\ref{metric}) of the pseudo-metric we discover
\begin{eqnarray*}
\lambda &=& h(\0_{x(s)} \oplus \dot {\bar x}(t_0), \dot x(s) \oplus \0_{\bar x(t_0)})
\\ &=& - \dot x^{\bar i}(t_0) \dot q^*_{\bar i}(s).
\end{eqnarray*}
Integrating this constant over $0<s<1$, (\ref{critical point}) yields the desired orthogonality
$$
\lambda =  \dot x^{\bar i}(t_0)(q^*_{\bar i}(0) - q^*_{\bar i}(1)) =0.
$$

Now (\ref{strict convexity}) shows
$\partial^2 g/\partial t^2|_{t=t_0}$ to be a convex function of $s \in [0,1]$.
We shall prove this convex function is minimized at $s=0$,  where it vanishes.  Introducing
coordinates $x^1,\ldots,x^n$ around $x=x(0)$ on $M$ and
$x^{\bar 1},\ldots,x^{\bar n}$ around $\bar x(t_0)$ on $\bar M$, we compute
\begin{eqnarray*}
\frac{\partial^2 g}{\partial t^2}\bigg|_{(s,t_0)}
&=& 
     -\Big[c_{{\bar i}}(x(s),\bar x(t)) \ddot x^{\bar i} + c_{{\bar i} {\bar j}}(x(s),\bar x(t)) \dot x^{\bar i} \dot x^{\bar j}
      \Big]^{(s,t_0)}_{(0,t_0)}
\\ \frac{\partial^3 g}{\partial s \partial t^2}\bigg|_{(s,t_0)}
&=& - (c_{{\bar i} k}(x(s),\bar x(t_0)) \ddot x^{\bar i} + c_{{\bar i} {\bar j} k}(x(s),\bar x(t_0)) \dot x^{\bar i} \dot x^{\bar j}) \dot x^k.
\end{eqnarray*}
When $s=0$,  the last line vanishes by the geodesic equation for
$t \in[0,1] \longrightarrow (x(0),\bar x(t))$, and the preceding line is
manifestly zero.  Thus the strictly convex function
$\partial^2 g/\partial t^2|_{t=t_0}$ must be nonnegative for $s \in [0,1]$
and, as initially claimed,  $\ddot f(t_0) = \partial^2 g/\partial t^2|_{(s,t)=(1,t_0)}$
is non-negative at any $t_0 \in \mathopen]0,1\mathclose[$ where $\dot f(t_0)=0$.

If the relevant cross-curvature of the cost is positive at one point $(x(s_0),\bar x(t_0))$,
then the non-negative function $\partial^2 g/\partial t^2|_{t=t_0}$ is strictly convex
(\ref{strict convexity}) on an interval around $s_0 \in [0,1]$;
since $\partial^2 g/\partial t^2|_{t=t_0}$ is minimized at $s=0$ it must then be
positive at $s=1$, to conclude the proof.
\end{proof}

\begin{definition}\label{D:visible set}{\bf (Illuminated set)}
Given $(x,\bar x) \in N$,  let $V(x,\bar x) \subset M$ denote those points
$y \in N(\bar x)$
for which there exists a curve from $(x,\bar x)$ to $(y,\bar x)$ in $N(\bar x) \times \{\bar x\}$
satisfying the geodesic equation on $(N,h)$.
\end{definition}

As a warm up to the {\em double mountain above sliding mountain} Theorem \ref{T:maximum principle},
let us derive a strong version of this result under the simplifying hypothesis
that the cost is strictly {\bf (A3s)} and not merely weakly regular {\bf (A3w)}.

\begin{corollary}{\bf (Strict maximum principle)}
Use a strictly regular cost $c \in C^4(N)$ to define a pseudo-metric
(\ref{metric}) on the domain $N \subset M \times \bar M$.  Let
$t\in[0,1] \longrightarrow (x, \bar x(t)) \in N$ be a geodesic in $(N,h)$ with
$\dot {\bar x}(0) \ne \0$.
Then for all $y \in \cap_{t\in[0,1]} V(x,\bar x(t))$ the function
$f(t) =-c(y,\bar x(t)) + c(x,\bar x(t))$ satisfies
$f(t) < \max\{f(0),f(1)\}$ on $0<t<1$.
\end{corollary}

\begin{proof}
Let $t\in[0,1] \longrightarrow (x, \bar x(t)) \in N$ be a geodesic in $(N,h)$ with
$\dot {\bar x}(0) \ne \0$.  Given $y \in \cap_{t\in[0,1]} V(x,\bar x(t))$ and
set $f(t) =-c(y,\bar x(t)) + c(x,\bar x(t))$. Notice $y,x \in N(\bar x(t))$
so $f$ is $C^4$ smooth on $[0,1]$.  Proposition \ref{P:critical implies min} asserts
$\ddot f (t_0)>0$ at each interior critical point $\dot f(t_0)=0$.
Any critical point of $f$ in $\mathopen]0,1\mathclose[$ is therefore a local
minimum,  and $f$ is strictly monotone away from this point.  Thus
$f(t) < \max\{f(0),f(1)\}$ for $0<t<1$, as desired.
\end{proof}

If $N$ is horizontally convex then $V(x,\bar x(t))=N(\bar x(t))$,  which motivates
the relation of this corollary to Theorem \ref{T:second stab}.
Let us now show a weak version of this maximum principle survives
as long as the cost is weakly regular.
To handle this relaxation we use a level set approach.

\begin{lemma}\label{L:level set}{\bf (Level set evolution)}
Let $g \in C^2 \left(\mathopen]\epsilon,\epsilon\mathclose[ \times U \right)$ where
$U \subset \Rn$ is open.  Suppose $D g = (\partial_1 g,\ldots, \partial_n g)$ is
non-vanishing on $\mathopen]\epsilon,\epsilon\mathclose[ \times U$. Then the
zero set $S(t) = \{ x \in U \mid g(t,x)=0\}$ is a $C^2$-smooth $n-1$ dimensional submanifold
of $U$ which can be parameterized locally for small enough $t$ by $\{ X(t,z) \mid z \in S(0)\}$,
where the Lagrangian variable $X(t,x)$ solves the ordinary differential equation
\begin{equation}\label{Lagrangian ODE}
\frac{\partial X(t,x)}{\partial t} = -
\left[\frac{\partial g}{\partial t} \frac{Dg}{|Dg|^2}\right]_{(t,X(t,x))}
\end{equation}
subject to the initial condition $X(t,x)=x$.
Moreover, the positivity set
$S^+(t) = \{x \in U \mid g(t,x) \ge 0\}$ has
$S(t)$ as its boundary,  and expands with an outward normal velocity given by
\begin{equation}\label{normal velocity}
v = - \frac{\partial X}{\partial t}\bigg|_{(t,z)} \cdot \frac{Dg}{|Dg|}\bigg|_{(t,X(t,z))}
= \frac{\partial g/\partial t}{|Dg|}\bigg|_{(t,X(t,z))}.
\end{equation}
\end{lemma}

\begin{proof}
Clearly the boundary of $S^+(t)$ is contained in $S(t)$.
Since $Dg \ne 0$,  the implicit function theorem implies $S(t)$ is a
$C^2$-smooth hypersurface and separates regions where $g(t,x)$ takes
opposite signs.  Thus $S(t)$ is contained in and hence equal to the boundary
in $U$ of the positivity set.  If the desired parametrization exists it must satisfy
$0=g(t,X(t,z))$.  Differentiation in time yields an equation
$$
0 = \frac{\partial g}{\partial t}(t,X(t,z))
+  Dg(t,X(t,z)) \cdot \frac{\partial X}{\partial t}\bigg|_{(t,z)}
$$
easily seen to be equivalent to (\ref{Lagrangian ODE}).  Conversely, near a
point $(0,z) \in \mathopen]-\epsilon,\epsilon\mathclose[ \times S(0)$, the
$C^1$ vector field (\ref{Lagrangian ODE}) can be integrated for a short time
(depending on $z$) to yield the desired parametrization.
\end{proof}

\begin{theorem}\label{T:maximum principle}{\bf (Double mountain above sliding mountain)}
Use a weakly regular cost $c \in C^4(N)$ to define a pseudo-metric
(\ref{metric}) on the domain $N \subset M \times \bar M$.  Let
$t\in]0,1[ \longrightarrow (x, \bar x(t)) \in N$ be a geodesic in $(N,h)$.
If $\mathopen]t_0,t_1\mathclose[\times \{y\}$ lies in the interior of
\begin{equation}\label{sliding domain}
\Lambda := \{ (t,y) \in \mathopen[0,1\mathclose] \times M \mid y \in V(x,\bar x(t)) \},
\end{equation}
then $f(t) =-c(y,\bar x(t)) + c(x,\bar x(t)) \le \max\{f(t_0^+),f(t_1^-)\}$
on $0\le t_0<t<t_1 \le 1$, where
\begin{equation}\label{discontinuous limits}
f(t_0^+) = \lim_{\epsilon \searrow 0} f(t_0 + \epsilon), \qquad
f(t_1^-) = \lim_{\epsilon \searrow 0} f(t - \epsilon).
\end{equation}
\end{theorem}

\begin{proof}
Fix a geodesic $t\in]0,1[ \longrightarrow (x, \bar x(t)) \in N$ with
$\dot{\bar x}(1/2) \ne \0$, since otherwise the conclusion is obvious.
Note that $\cexp_x$ and hence $t \in \mathopen]0,1[ \longrightarrow {\bar x}(t)$ are
$C^3$ smooth,  from Lemma \ref{L:c-segments are geodesics} and {\bf (A0)}.
For all $y \in M$ and $t \in \mathopen]0,1[$ set
$$
f(t,y) = -c(y,\bar x(t)) + c(x,\bar x(t))
$$
and note that $f(t,y)$ is $C^3$-smooth on the interior of $\Lambda \subset [0,1] \times M$.
Define
\begin{eqnarray*}
S^+ &:=& \{ (t,y) \in \intr \Lambda \mid \frac{\partial f}{\partial t} \ge 0\}
\\S &:=& \{ (t,y) \in \intr \Lambda \mid \frac{\partial f}{\partial t} = 0\}.
\end{eqnarray*}
For each $t\in \mathopen]0,1[$,
we think of $f(t,y)$ as defining the elevation of a landscape over
$M$, which evolves from $f(y,0)$ to $f(y,1)$ as $t$
increases, and is normalized so that $f(x,t)=0$.
We picture $f(y,t)$ as a sliding mountain, with
$S^+(t) := \{y \in M \mid (t,y) \in S^+\}$ denoting the rising region,
and $S(t):= \{y \in M \mid (t,y) \in S\}$ the region at the boundary of $S^+(t)$ which
--- instantaneously --- is neither rising nor sinking.

We claim the rising region $S^+(t) \subset M$ is a non-decreasing function of
$t \in \mathopen]t_0,t_1\mathclose[$.
To see this, we plan to apply Lemma \ref{L:level set} to the
$C^2$ function
\begin{equation}\label{level set function}
g(t,y)
:= \frac{\partial f}{\partial t}
= -\bar D c(y,\bar x(t))\dot {\bar x}(t) + \bar Dc(x,\bar x(t)) \dot {\bar x}(t)
\end{equation}
on $\Lambda$. Differentiating this function with respect to $y \in M$ yields
\begin{eqnarray*}
Dg(t,y) &=& - D \bar D c(y,\bar x(t)) \dot {\bar x}(t)
\\ &\ne& \0
\end{eqnarray*}
which is non-vanishing because of {\bf (A2)}.  Applying Lemma \ref{L:level set}
on any coordinate chart in $M$ shows $S(t)$ is the boundary of
$S^+(t)$,  and the question of whether $S^+(t)$ is expanding or
contracting along its boundary is determined by the sign of
$\partial g/\partial t$ on $S(t)$.

 From the definition of $V(x,\bar x(t'))$ observe
$y \in S(t') \subset V(x,\bar x(t'))$ implies $(y,\bar x(t'))$ is linked to
$(x,\bar x(t'))$ by a curve in $N(\bar x(t')) \times \{\bar x(t')\}$ which is geodesic
in $(N,h)$.  Proposition \ref{P:critical implies min} asserts
$\partial^2 f/\partial t^2|_{(t',y)} \ge 0$,  so from Lemma \ref{L:level set} there is
a neighbourhood of $y$ on which the rising region $S^+(t)$ does not shrink for a short
time interval around $t' \in ]0,1[$.

Finally,  fix $y \in M$ and $0\le t_0<t_1 \le 1$
such that $\mathopen]t_0,t_1\mathclose[ \times \{y\} \subset \intr\Lambda$
and let $U \subset \mathopen]t_0,t_1\mathclose[$ denote the
open set of times at which $y \not\in S^+(t)$.
If $U$ is non-empty,
we claim any connected component
of $U$ has $t=t_0$ in its closure.  If not,  let $t' \in \mathopen]t_0,t_1\mathclose[$
denote the left endpoint of a connected component in $U$.  This means $y \in S^+(t')$
but $y \not\in S^+(t'+\delta)$ for any $\delta >0$, in violation of the non-shrinking
property of $S^+(t)$ derived above.  Thus 
$t \in \mathopen]t_0,t_1\mathclose[ \longrightarrow f(t,y)$ is decreasing on an
interval $U= \mathopen]t_0, t(y)\mathclose[$ for some $t(y) \in [t_0,t_1]$
and non-decreasing on the complementary interval $]t(y),t_1\mathclose[$.
The limits (\ref{discontinuous limits}) exist and the proof that
$f(t,y) \le \max\{f(t_0^+),f(t_1^-)\}$ is complete.
\end{proof}

\begin{remark}[A geodesic hypersurface bounds the rising region]
Lemma \ref{L:c-segments are geodesics} implies
$S(t) \times \{\bar x(t)\}$ is a totally geodesic submanifold for each $t \in ]0,1[$
of the preceding proof.  Indeed, (\ref{level set function}) shows
$\bar q^* \in \dom (\cdexp_{\bar x(t)}) \subset T^*_{\bar x(t)} \bar M$
belongs to the hyperplane $( \bar D c(x,\bar x(t)) + \bar q^*) \dot {\bar x}(t)=0$,
if and only if $y := \cdexp_{\bar x(t)} \bar q^*$ lies on $S(t)$.
\end{remark}

\begin{remark}\label{R:Lambda is open}
If the domain $N$ is horizontally convex then $V(x,\bar x) = N(\bar x)$ and $\Lambda$ from
(\ref{sliding domain}) are open sets.  Then $]0,1[ \times \{y\} \in \Lambda$ if and only if
$y \in \cap_{0<t<1} N(\bar x(t))$.
\end{remark}

\begin{remark}[Enhancements]
Villani has incorporated a version of our proof into \cite{Villani06},
adding simplifications which allow him to avoid the use of the level set method.
He further develops our technique to prove additional results: contrast our
Theorem \ref{T:supporting c-convex},  which gives a new approach to a result
of Loeper, with Villani's Theorem 12.41, 
which gives a new approach to a result of Trudinger \& Wang.
\end{remark}

\section{Perspective and conclusions}

Before concluding this paper,  let us briefly review the connection of
the transportation problem (\ref{Kantorovich}) we study with fully
non-linear partial differential equations.  Although this connection goes
almost back to Monge \cite{Monge81},  it has developed dramatically
since the work of Brenier \cite{Brenier87}.
When an optimal mapping $F:M \longrightarrow \bar M$ exists and happens
to be a diffeomorphism,  it provides
a change of variables between $(M,\rho)$ and $(\bar M,\bar \rho)$, hence
its Jacobian satisfies the equation
\begin{equation}\label{prescribed determinant}
 \bar \rho(F(x)) |\det DF(x)| = \rho(x).
\end{equation}
Often when $F$ is not smooth,
(\ref{prescribed determinant}) remains true almost everywhere
\cite{McCann95} \cite{Cordero-Erausquin04} \cite{AmbrosioGigliSavare05}.

Optimality implies that $F(\cdot)$ can be related to pair of scalar functions
$u:M \longrightarrow \R \cup \{+\infty\}$ and
$\bar u: \bar M \longrightarrow \R \cup \{+\infty\}$,
which arise from the linear program dual to (\ref{Kantorovich}),
and represent Lagrange multipliers for the prescribed densities $\rho$ and $\bar \rho$.
Moreover, $u$ can be taken to belong to the class of {\em $c$-convex} functions,
meaning it can be obtained as the upper envelope of a family of mountains --- analogous to
(\ref{double mountain}) and (\ref{triple peak}) but with $i$ 
ranging over $\bar \rho$-almost
all of $\bar M$;  similarly, $\bar u$ is a $c^*$-convex function, as in
Definition \ref{D:c-convex function} below.
When the twist condition holds, $F = \cexp \circ Du$.  If non-degeneracy {\bf (A2)} also
holds,  then (\ref{prescribed determinant}) becomes an equation of Monge-Amp\`ere
type expressed as in \cite{MaTrudingerWang05}
using local coordinates around $(x',F(\bar x'))$ by
\begin{equation}\label{Monge-Ampere type}
\det [u_{ij} + c_{ij}]_{(x,c{\rm -Exp}_x Du(x))} =
\frac{\rho(x)}{\bar \rho(\cexp_x Du(x))} |\det 
c_{i\bar k} |_{(x,c{\rm -Exp}_x Du(x))};
\end{equation}
$c$-convexity of $u$ implies non-negative definiteness of the matrix
 $u_{ij} + c_{ij}$ of second derivatives,
hence the equation is degenerate elliptic \cite{GangboMcCann96}.
Without further assumptions on $(M,\rho)$ and the geometry of $(\bar M,\bar \rho)$
we can expect neither strict ellipticity nor regularity of solutions:
if the support of $\rho$ is connected but
the support of $\bar \rho$ is not,  this will force $F$ to be discontinuous
and $u \not\in C^1(M)$.  The equation is therefore not locally smoothing, and the best
one can hope is for solutions to inherit regularity from the boundary data
$(M,\rho)$ and $(\bar M,\bar \rho)$.

For the cost function $c(x,\bar x)=|x-\bar x|^2/2$ on $M, \bar M \subset \Rn$,
(\ref{Monge-Ampere type}) becomes the familiar Monge-Amp\`ere equation for the
convex function $u(x) + |x|^2/2$ \cite{Brenier87}.  In this case
Caffarelli \cite{Caffarelli92} was able to show {\em convexity} of $\bar M$ implies
local H\"older continuity of $F$ if the densities
${d\rho}/{d\vol} \in L^\infty(M)$ and ${d\vol }/{d\bar \rho} \in L^\infty(\bar M)$
are bounded,  and smoothness of $F$ on the interior of $M$
if the densities $\log|d\rho / d\vol|$ and $\log|d\bar \rho / d\vol|$ are bounded and smooth;
see also Delano\"e \cite{Delanoe91} for the case $n=2$,  and
Wang \cite{Wang96} for analogous results in the context of
Example \ref{E:reflector antenna}, the reflector antenna problem.
The results of Loeper and Ma, Trudinger \& Wang
extend the H\"older \cite{Loeper07p} and smooth  \cite{MaTrudingerWang05} \cite{TrudingerWang08p}
theories to general bi-twisted, strictly regular costs on $\cl (M \times \bar M)$.
Loeper in particular achieves stronger results such as a global H\"older estimate
with explicit exponent under weaker restrictions on $\rho$ and $\bar \rho$
by exploiting strict regularity of the cost;
see Theorems \ref{T:Hoelder} and \ref{T:regular sphere}.
Trudinger \& Wang extended the
up-to-the-boundary regularity results of Caffarelli \cite{Caffarelli96}
and Urbas \cite{Urbas97} --- which require convexity and smoothness of $M\subset \Rn$ as well as
$\bar M$ --- to bi-twisted costs which are merely weakly regular \cite{TrudingerWang07p}.
For bounded and sufficiently smooth densities,  horizontal and vertical convexity of
$N = M \times \bar M$
takes the place of the convexity assumptions on the target $\bar M$ and domain $M$
in these theories.

Since we had not previously encountered pseudo-Riemannian geometry of any signature other
than the Lorentzian one $(n,1)$ in applications,  much less as the necessary and sufficient
condition for degenerate elliptic partial differential equations to possess smooth solutions,
a few words of explanation seem appropriate.\footnote{
However, as we learned subsequently from Robert Bryant,
the wedge product $\omega \wedge \omega$ can be used to define a signature $(3,3)$
pseudo-metric on the space $\R^4 \wedge \R^4$;  in four dimensions,  the difference between
\emph{positive sectional curvature} and \emph{positive curvature operator} amounts to
the question of whether the curvature operator is positive-definite only on the light cone
with respect to this pseudo-metric,  or on the full space; c.f.\ \cite{Besse87}.
In one way this parallels the distinction between strict regularity and
positive cross-curvature of a cost;   in another it parallels the
distinction between positive cross-curvature 
and positive sectional curvature,
(\ref{full cross-curvature})--
(\ref{equivalent to sectional curvature non-negativity}).}

The regularity of optimal maps $F:M \longrightarrow \bar M$ is a question whose
answer should depend only on the cost function $c(x,\bar x)$
and the probability measures $\rho$ and $\bar \rho$.  It should not depend on
which choice of smooth coordinates on $M$ and $\bar M$ are used to represent
this data or the solution. Any necessary and sufficient condition
on $c(x,\bar x)$ guaranteeing regularity of $F$ should therefore be geometrically invariant,
meaning coordinate independent.
Thus pseudo-Riemannian geometry and curvatures arise in the theory
of optimal transportation for the same reason they arise in
Einstein's theory of gravity, {\em general relativity}:  they provide the natural language for describing
phenomena --- in this case regularity --- which exhibit invariance under the general group
of diffeomorphisms.
Put another way,  the underlying physical reality is independent of
whose coordinates are used to describe it.  Moreover, this invariance places
severe restrictions on the form which necessary and sufficient conditions
for regularity can take:
since the pseudo-metric $h$ is equivalent to knowing the cost function
$c(x,\bar x)$ --- up to null Lagrangians $v(x) + \bar v(\bar x)$ which are
irrelevant to the optimization at hand (\ref{Kantorovich}) ---  any such conditions
on the cost function must be expressible via the curvature tensor of $h$.

Let us turn now to the question of why the intrinsic geometry of optimal transportation
should be pseudo-Riemannian rather than Riemannian.  Dimensional symmetry between the
domain $M$ and target $\bar M$ suggest that the number of time-like directions in the
theory --- if any --- should equal the number of space-like directions.  But why
signature $(n,n)$ rather than the signature $(2n,0)$, which is more frequently associated to
elliptic and extremal problems? And why the nullity of $p \oplus \bar p$ in {\bf (A3w)}?

The Riemannian notion of length allows us to associate a magnitude to any sectional
curvature.  However,  null vectors have no length and cannot be
normalized; because the plane $(p \oplus \0) \wedge (\0 \oplus \bar p)$ is generated by
orthogonal null-vectors, we can decide the sign (positive, negative, or zero) of its sectional curvature,
but not the magnitude.  On the other hand, the results
of Loeper reveal that the size of the constant $C>0$ in hypothesis (\ref{TWA3w})
controls the H\"older constant of the mapping $F:M \longrightarrow \bar M$.
Unlike the exponent, which is coordinate-independent,  the H\"older constant of $F$
obviously depends on the choice of coordinates.  We are therefore relieved to find
the cross-curvature condition governing regularity does not have an associated magnitude,
since the problem has no intrinsic length scale.
To be scale free, the geometrical structure which governs regularity
for optimal transportation must be pseudo-Riemannian,
since the modulus of continuity of a map $F$ has no
intrinsic meaning in the absence of separate notions of length on $M$ and $\bar M$,
which the cost function $c(x,\bar x)$ alone cannot provide.  What it can and does provide
are geodesics on $N(\bar x) \times \{\bar x\}$ and $\{x\} \times N(x)$, and
geodesic convexity of these null submanifolds is the essential 
domain hypothesis in Ma, Trudinger \& Wang's theory.

In the present investigation,  we have focused exclusively on the pseudo-metric
$d\ell^2 = -c_{i\bar j} dx^i dx^{\bar j}$ induced by the cost (\ref{metric}).
Let us conclude by noting that there is also a canonical symplectic form
$\omega = d(D c \oplus \0) = -d(\0 \oplus \bar Dc)$ on $N \subset M \times \bar M$
associated to the cost $c \in C^4(N)$. In local coordinates $x^1,\ldots,x^n$ on $M$
and $x^{\bar 1},\ldots, x^{\bar n}$ on $\bar M$ it is given by
\begin{equation}\label{symplectic form}
\omega :=  \frac{1}{2}\left(
\begin{array}{cc}
0 & \bar D D c \\
-D \bar D c & 0
\end{array}
\right).
\end{equation}
It is possible to verify that any $c$-optimal diffeomorphism
$F:M \longrightarrow \bar M$ has a graph which is spacelike with respect to $h$
and Lagrangian with respect to $\omega$. Conversely,
for a weakly regular cost, results of Trudinger \& Wang \cite{TrudingerWang07p}
or Villani \cite{Villani06} can be used to deduce
that any diffeomorphism whose graph is $h$-spacelike and $\omega$-Lagrangian is in fact the
$c$-optimal map between the measures $\rho := \pi_\#(\vol^{(N,h)}|_{\graph(F)})$ and
$\bar \rho := \bar \pi_\# (\vol^{(N,h)}|_{\graph(F)})$ obtained by projecting the
Riemannian volume $\vol^{(N,h)}$ induced by $h$ on $\graph(F)$
through the canonical projections $\pi(x,\bar x) = x$ and $\bar \pi(x,\bar x) = \bar x$.
This reveals another unexpected connection between optimal transportation and symplectic
(or pseudo-K\"ahler) geometry.
When $\rho$ and $\bar \rho$ are given by the Euclidean volumes on two convex domains,
and $c(x,\bar x) = |x-\bar x|^2/2$,  this is related to the work of
Wolfson \cite{Wolfson97} and Warren \cite{Warren} on special Lagrangian submanifolds,
where a pseudo-Riemannian metric of signature $(n,n)$ also appears \cite{Warren}.
We defer the details of this development to a future work.

\appendix

\section{Regularity consequences for optimal transportation}
\label{S:overview}

This series of appendices is devoted to explaining how the results obtained above
combine with Loeper's ideas  to simplify the proof of
his H\"older continuity results for optimal mappings with respect to
cost functions which are strictly regular {\bf (A3s)} --- meaning 
they have positive cost-sectional curvature in the language of Loeper
\cite{Loeper07p}. 
To emphasize that this argument is completely self-contained,
we recall the necessary details of his proof in full,
exploiting where possible the conceptual framework developed above.
We often present variations on his arguments,
but do not claim originality for the conclusions.

The plan can be outlined as follows.  In the present section we recall
the relevant facts of life from the literature,  and
derive the key local-implies-global ingredient ---  a variant of which appears
in Trudinger \& Wang \cite{TrudingerWang08p}
--- as a simple corollary to our main theorem
above.  In Appendix \ref{S:Taylor} we recall Loeper's Taylor expansion, which uses strict
positivity of orthogonal cross-curvatures to quantify the altitude of the double
mountain over the sliding mountain near the point in the valley of indifference
where the sliding mountains are normalized to coincide.
Appendices \ref{S:sausage} and \ref{S:volume comparison} explain how Loeper used this result
to obtain H\"older continuity of optimal mappings between suitable measures on Euclidean
domains.  Appendix \ref{S:spherical continuity} employs the same argument in a simpler
setting to obtain continuity of optimal mappings between suitable measures on the round
sphere,  when optimality is measured either against (a) Riemannian distance squared
\cite{Loeper07p},
or (b) the reflector antenna cost function, as in \cite{Loeper07p}
and Caffarelli, Guti\'errez \& Huang \cite{CaffarelliGutierrezHuang}.
As a corollary,  we recover
directly a bound on how closely an optimal map may approach the cut locus, which
can be used in place of Delano\"e \& Loeper's estimate \cite{DelanoeLoeper06}
in the next step. 
Once the continuity of the optimal map has been established,
Caffarelli, Guti\'errez, Huang and Loeper's
H\"older regularity results concerning
optimal maps on the spheres (a) \cite{Loeper07p} and (b) \cite{CaffarelliGutierrezHuang}
\cite{Loeper07p}
can be recovered using bi-twisting {\bf (A1)} alone without further resort to the
local-implies-global theorem:  the problem is easily localized when continuity is known.
This approach is considerably more direct than Loeper's original argument,
which established the local-implies-global theorem using approximation by
smooth solutions to a family of auxiliary problems constructed by combining Delano\"e's
continuity method \cite{Delanoe04} with Delano\"e \& Loeper's control on proximity to the
cut locus \cite{DelanoeLoeper06},
and the a priori estimates of Ma, Trudinger \& Wang \cite{MaTrudingerWang05}.

\subsection{$c$-convex functions and their properties}

The definition of $c$-convex functions and some basic
properties are discussed.  We first introduce the notion of a supporting mountain
(or $c$-mountain),  also called $c$-support functions, $c$-planes or $c$-graphs.
The focus (or $c$-focus) of such a mountain is defined below,  and was called
the $c$-normal by Trudinger \& Wang, in analogy with the Gauss map and generalized normal
in the theory of convex bodies.

We shall find it convenient to continue to assume $M$ and $\bar M$ are open manifolds
whose closures $\cl M$ and $\cl \bar M$ are compact in some larger space,
while $N \subset M \times \bar M$
intersects neither $M \times \partial \bar M$ nor $(\partial M) \times \bar M$.

\begin{definition}{\bf (Supporting mountains)}
Fix $c:\cl(M \times \bar M) \longrightarrow \R \cup \{+\infty\}$.
A {\em mountain} refers to any function $f$ on $\cl M$
of the form
$$
f(\cdot) = -c(\cdot,\bar x) + \lambda
$$
with $(\lambda,\bar x) \in \R \times \cl \bar M$.  The mountain is said to
be {\em focused} at $\bar x$.  The mountain $f$ is said to
{\em support} $u:\cl M \longrightarrow \R \cup\{+\infty\}$ at $x \in \cl M$
if $u(x)=f(x)< +\infty$ and
$$
u(y) \ge f(y)
$$
for all $y \in \cl M$.
The mountain is said to support $u$ \emph{to first order} at
$x \in M$
if $u(x) = f(x)<+\infty$, and $M$ has a manifold structure near $x$ with
\begin{align}
\label{first-order support}
u(y) \ge f(y) + o(|y-x|) \qquad {\rm as}\ y \to x.
\end{align}
\end{definition}

\begin{remark}[Twisting identifies focus]\label{R:twist identifies focus}
If a mountain $f$ supports $u$ to first order at a point $x\in M$ where $u$ happens
to be differentiable,  and $f$ is known to be focused in $\bar N(x)$,  then
twisting {\bf (A1)} identifies the focus $\bar x = \cexp_x Du(x)$ uniquely,
since \eqref{first-order support} then implies $Du(x) = Df(x)= -Dc(x,\bar x)$.
\end{remark}



\begin{definition}{\bf ($c$-convex function focused in $\bar \Omega$)}\label{D:c-convex function}
Fix $c:\cl(M \times \bar M) \longrightarrow \R \cup \{+\infty\}$ and a
subset 
$\bar \Omega \subset \cl \bar M$. Then
$u:\cl M \longrightarrow \R \cup\{+ \infty\}$ belongs to
$\mathcal{S}^{-c}_{\bar \Omega}$ 
if $\dom u := \{x \in M\mid u(x)<+\infty\}$ is non-empty
and there exists $v:\bar \Omega \longrightarrow \R \cup\{+\infty\}$ such that
$u= v^c_{\bar \Omega}$, where
\begin{align}
v^c_{\bar \Omega}(x)
&:= \sup_{\bar x \in \bar \Omega} -c(x,\bar x)-v(\bar x).
\nonumber
\end{align}
\end{definition}

Functions in the class $\mathcal{S}^{-c}_{\bar \Omega}$ are sometimes called \emph{$c$-convex},
though they might also be called \emph{$(-c)$-convex}, since they are
\emph{supremal convolutions} with $-c$.  Our subscript keeps track of the set
$\bar \Omega$ of foci;  we say $u \in \mathcal{S}^{-c}_{\bar \Omega}$
is a $c$-convex function focused in $\bar \Omega$.

\begin{lemma}{\bf (Properties $c$-convex functions inherit from the cost)}
\label{L:inherit}
Fix metric spaces $(M,d)$ and $(\bar M,\bar d)$, a cost
$c:\cl(M \times \bar M) \longrightarrow \R \cup \{+\infty\}$
and subset $\bar \Omega \subset \cl \bar M$.  Recall $c^*(\bar x, x) := c(x,\bar x)$.
\\{\em (a) Legendre-type duality:}
If $u \in \mathcal{S}^{-c}_{\bar \Omega}$ then $u = (u^{c^*}_{\cl M})^c_{\bar \Omega}$.
\\ {\em (b) Semicontinuity:} If $c:\cl M \times \bar \Omega \longrightarrow \R$ is continuous and
$u \in \mathcal{S}^{-c}_{\bar \Omega}$, then $u: \cl M \longrightarrow \R \cup \{+\infty\}$
is lower semicontinuous and $\partial^c u \subset \cl(M \times \bar M)$  is closed.
\\{\em (c) Lipschitz:}
If $\Omega \subset \cl M$ and 
$u \in \mathcal{S}^{-c}_{\bar \Omega}$ then 
\begin{align*}
\Lip(u,\Omega)
:= \sup_{\Omega \ni y \ne z \in \Omega} \frac{u(y) - u(z)}{d(y,z)} 
&\le \sup_{\bar x \in \bar \Omega} \Lip(-c(\cdot,\bar x),\Omega).
\end{align*}
{\em (d) Semiconvex:}
If $(M,d) = (\R^n,|\cdot|)$ and
$D^2 c(\cdot,\bar x)\le C I$ holds
in the (matrix) sense of distributions on a domain $\Omega \subset \Rn$,
with $C=C(\Omega, \bar \Omega) \in \R$ independent of
$\bar x \in\bar \Omega$,
then each $u \in \mathcal{S}^{-c}_{\bar \Omega}$
satisfies $D^2 u \ge -CI$ in the 
sense of distributions on $\Omega$.
\end{lemma}

\begin{proof}
All four facts are well-known;  (b) is elementary to check and the other
three statements are proved, e.g., in \cite{Villani06}, where (a) is Proposition 5.8,
and (c) and (d) are established in the proof of Theorem 10.26.
\end{proof}

\begin{corollary}{\bf (Lipschitz/semiconvex)}\label{C:Lipschitz/semiconvex}
If $c$ is a locally Lipschitz function on the (compact) closure of
$\cl(M \times \bar M)$,
then $u \in \mathcal{S}^{-c}_{\cl \bar M}$ is locally Lipschitz on $\cl M$
and $\cl M \subset \dom \partial^c u$.
If, in addition, 
$c$ is locally semiconcave,
then $u \in \mathcal{S}^{-c}_{\cl \bar M}$ is locally semiconvex
(\ref{D:semiconvex}) on $M$. 
\end{corollary}

\begin{proof}
The fact that $u \in \mathcal{S}^{-c}_{\cl \bar M}$ inherits Lipschitz or semiconvexity
properties from the cost follows directly from Lemma \ref{L:inherit}(c) and (d), using
compactness of $\cl \bar M$ to conclude that the local semiconcavity constant $C$ of
$c(\cdot, \bar x)$ and Lipschitz constant $\Lip(-c(\cdot,\bar x),\Omega)$ are
independent of $\bar x \in \cl \bar M$.  To see $\cl M \subset \dom \partial^c u$,  use
Lemma \ref{L:inherit}(a) to write
\begin{equation*}
u(x) = \sup_{\bar x \in \cl \bar M} -c(x,\bar x) - u^{c^*}_{\cl M}(\bar x).
\end{equation*}
Here $u^{c^*}_{\cl M}$ is lower semicontinuous by Lemma \ref{L:inherit}(b)
so the supremum is attained at some $\bar x \in \cl \bar M$, whence
\begin{align*}
u(x) + c(x,\bar x) &= - u^{c^*}_{\cl M} (\bar x)
\\ &=- \sup_{y\in \cl M} - c(y,\bar x) - u(y)
\\ & \le u(y) + c(y,\bar x)
\end{align*}
for all $y \in \cl M$.
Thus $(x,\bar x) \in \partial^c u$ as desired (\ref{c-subdifferential}).
\end{proof}

Note that $c$ is both locally Lipschitz and semiconcave if it is weakly regular
on $\cl(N)$, or else if $c=d^2/2$ is the distance squared on a compact Riemannian manifold
$(M=\bar M, g)$ without boundary as in Example \ref{E:Riemannian};  see Cordero-Erausquin,
McCann, Schmuckenschl\"ager \cite{CorderoMcCannSchmuckenschlager01}.
This is not the case for the reflector antenna cost function of
Example \ref{E:reflector antenna}, but the associated $c$-convex
functions are Lipschitz and semiconvex anyways --- unless they are
mountains --- as the following proposition shows.  Here,  as above,
semiconvexity of $u:M \longrightarrow \R$
means that near each point $z \in M$ there is a coordinate ball $B_r(z) \subset M$
and coordinate-dependent constant $C<\infty$ such that $D^2 u(x) \ge -C\I$
as matrices distributionally in these coordinates,  or equivalently that
\begin{equation}\label{D:semiconvex}
x \in B_r(z) \longrightarrow u(x) + C|x|^2/2
\end{equation}
is a convex function,  where $|x|$ denotes the Euclidean norm of the coordinates.
Semiconcavity of $u$ means $-u$ is semiconvex.

\begin{proposition}{\bf (Semiconvexity of reflector antenna potentials)}
\label{P:semi-convex reflector antenna}
Fix the cost $c(x,\bar x) = -\log|x -\bar x|$ on the Euclidean unit sphere
$\Sphere^n = \partial B_1(\0) \subset \R^{n+1}$. If $u:\Sphere^n \longrightarrow \R \cup \{+\infty\}$
belongs to $\mathcal{S}^{-c}_{\Sphere^n}$ and
takes finite values at two or more points, then $u$
is locally Lipschitz and semiconvex on
$\Sphere^n$.  In coordinates, its Lipschitz and semiconvexity constants do not depend on $u$
except through $\max_{\Sphere^n} u - \min_{\Sphere^n} u$.
\end{proposition}

\begin{proof}
Since $u \in \mathcal{S}^{-c}_{\Sphere^n}$, defining the dual potential
$\bar u(\cdot)= \sup_{x \in \Sphere^n} \log | \cdot - x| -u (x)$ yields
\begin{align}\label{Legendre duality}
u(\cdot) = \sup_{\bar x \in \Sphere^n} \log|\cdot -\bar x| - \bar u(\bar x)
\end{align}
similarly to Lemma \ref{L:inherit}(a).
We shall show that both $u$ and $\bar u$ are bounded from below and above.
Since $u(\cdot)$ takes finite values at $x_0 \ne x_1$ in $\Sphere^n$, we have
$\bar u(\cdot) \ge \max_{i=0,1} \log|\cdot -x_i| - u(x_i) 
\ge const > -\infty$,  where the constant depends only on $\max\{u(x_0),u(x_1)\}$ and
$|x_0-x_1|$.
Thus $\bar u(\cdot)$ is bounded below,
hence $u(\cdot)$ is bounded above from \eqref{Legendre duality}
since $|x-\bar x| \le 2$ for all $x,\bar x \in \Sphere^n$.

On the other hand,  if $\dom \bar u := \{\bar x \in \Sphere^n \mid \bar u(\bar x)<\infty\}$
consists of one or fewer points $\bar y$,  we reach the contradiction
$u(\bar y) = -\infty$.  So $\dom \bar u$ consists of two or more points, and
the same argument as before yields
$u(\cdot)$ bounded below and $\bar u(\cdot)$ bounded above.

To show $u$ is semiconvex,  recall that any supremum of smooth
functions which have locally uniform control on their
$C^2$ norms is semiconvex,  as in Lemma \ref{L:inherit}(d).
Let $-\bar x$ denote the antipodal point to a fixed point $\bar x \in \Sphere^n$.
If $|\lambda_0-\lambda_1|\le \Lambda$,  the function
$$
v(\cdot) := \max\{\log |\cdot - \bar x| - \lambda_0, \log|\cdot + \bar x| - \lambda_1\}
$$
is semiconvex on $\Sphere^n$,  with a semiconvexity constant $C$ depending only on
$\Lambda<+\infty$, but independent of $\lambda_0,\lambda_1 \in \R$ and $\bar x \in \Sphere^n$.
Setting $\Lambda = (\sup_{\Sphere^n} u) - (\inf_{\Sphere^n} u)$,  we see from \eqref{Legendre duality}
that
\begin{align*}
u (\cdot)
&= \sup_{\bar x \in \Sphere^n} \max\{\log|\cdot - \bar x| - \bar u (\bar x), \log |\cdot + \bar x| - \bar u(-\bar x)\}
\end{align*}
is a supremum of such functions $v(\cdot)$,  hence semiconvex with constant $C$ as desired.
Finally,  any bounded convex function is locally Lipschitz,  hence bounded semiconvex
functions are locally Lipschitz by (\ref{D:semiconvex}).
%
\end{proof}

\begin{definition}{\bf (Super- and subdifferential)}
The {\em subdifferential}
$\partial u \subset T^*M$ of a function $u:M \longrightarrow \R \cup \{+\infty\}$
on a smooth manifold is defined to
consist of those points $(x,q^*) \in T^*M$ in the cotangent bundle such that
in local coordinates around $x$,
\begin{equation}\label{subdifferential}
u(y) \ge u(x) + q^*(y-x) + o(|y-x|) \quad {\rm as}\ y \to x.
\end{equation}
\end{definition}
It follows that $\partial u(x) := \{q^* \in T_x^*M \mid (x,q^*) \in \partial u\}$
is a convex set,  and $\partial u(x) = \{Du(x)\}$ at those points
$x \in \dom Du \subset M \setminus \partial M$
where $u$ is differentiable.  The superdifferential $-\partial (-u)$ consists of
those pairs $(x,q^*) \in T^*M$ which satisfy the opposite inequality.
We set $\dom \partial u := \{x \in M \mid \partial u(x) \ne \emptyset\}$, so that
$\dom Du = (\dom \partial u) \cap (\dom \partial (-u))$.
If a function $u$ on $M$ is locally Lipschitz and semiconvex,  then
$\partial u \subset T^* M$ is closed and $M \subset \dom \partial u$.

%
%

Let us now hypothesize an extension of the twist condition to mountains focused
in $\cl(\bar M)$ as follows.

\begin{definition}{\bf (Extended twist)} A cost
$c:\cl(M \times \bar M) \longrightarrow \R \cup \{+\infty\}$ satisfies the extended
twist condition {\bf (A1$'$)} if each mountain $f$ 
is subdifferentiable on $M \cap \dom f$,
where $\dom f := \{y \in \cl M \mid f(y) <\infty\}$, and if
there is a continuous map $\cexp':T^*M \cap \cl (\dom \cexp) \longrightarrow \cl \bar M$,
such that each $x \in M$, mountain $f(\cdot) = - c(\cdot,\bar x) + \lambda$, and extreme point
$p^*$ of $\partial f(x)$ satisfy $(x,p^*) \in \cl (\dom \cexp)$ and $\bar x = \cexp'_x p^*$.
\end{definition}

One way to satisfy {\bf (A1$'$)} is to assume {\bf (A1)} extends to
some larger manifold $\tilde N$ in which $M \times \bar M \subset \subset \tilde N$
is compactly embedded,  which we abbreviate by saying {\bf (A1)} holds on $\cl N$.
Property {\bf (A1$'$)} is also verified for the examples discussed hereafter: any 
compact Riemannian manifold $(M=\bar M, g)$ with the geodesic distance squared cost
$c=d^2/2$ of Example \ref{E:Riemannian};
and the reflector antenna cost
$c(x,\bar x) = -\log |x - \bar x|$ on the sphere $\Sphere^n \times \Sphere^n$
as in Example \ref{E:reflector antenna}.
The key to verification in the Riemannian case is that the Riemannian exponential
is continuously defined throughout the tangent bundle, the cut-locus has no interior,
and mountains are semiconvex \cite{CorderoMcCannSchmuckenschlager01},  so
extreme points in $\partial f(x)$ of a mountain $f$
are limits of gradients $p^* = \lim_{k \to \infty} Df(x_k)$
from nearby points of differentiability $x_k \in \dom Df$; c.f.\ \S 25.6 \cite{Rockafellar72}.

The extended twist condition {\bf (A1$'$)} yields the following lemma and theorem.


\begin{lemma}\label{L:supporting c-convex}
{\bf (Tangent mountains support $c$-convex functions globally)}
Suppose a cost $c:\cl(M \times \bar M) \longrightarrow \R \cup \{+\infty\}$
twisted on $N \subset M \times \bar M$
satisfies the extended twist hypothesis {\bf (A1$'$)}.
If $u \in \mathcal{S}^{-c}_{\cl \bar M}$ is differentiable at
$x \in M \cap \dom \partial^c u$ then $(x,Du(x)) \in \cl \dom \cexp$ and
$\partial^c u(x) = \{\bar x\}$ with $\bar x := \cexp_x' Du(x)$.
Any mountain supporting $u$ to first order at $x$ is
focused at $\bar x$ and supports $u$ globally.
\end{lemma}

\begin{proof}
Fix a point $x \in M \cap \dom \partial^c u$
where $u \in \mathcal{S}^{-c}_{\cl \bar M}$ is differentiable.
Then $u(x) < +\infty$ and
there exists $\bar x \in \partial^c u(x)$ such that the mountain
$f(\cdot) = -c(\cdot,\bar x) + \lambda$ with $\lambda = c(x,\bar x) +u(x)$
supports $u$ at $x$.  Now let $g(\cdot) = -c(\cdot, \bar y) + \mu$ be
any other mountain which supports $u$ to first order at $x$.  Then
$g(y) \le u(y) +o(|y-x|) = u(x) + Du(x)(y-x) + o(|y-x|)$ is
superdifferentiable at $x$;  since $g$ is subdifferentiable by hypothesis,
we conclude $Dg(x)=Du(x)$.  In this case $\partial g(x) = \{Du(x)\}$,
has $Du(x)$ as an extreme point, so the extended twist hypothesis
implies $g$ is focused at $\bar y = \cexp_x' Du(x)$.
Applying the same argument to $f$ instead of $g$ yields $\bar x = \cexp_x'Du(x)$,
hence $\bar x = \bar y$ and $g=f + const$.  The constant vanishes since $g(x)=f(x)$,
so we conclude any mountain $g$ which supports
$u$ to first order at $x \in \dom Du$ actually supports $u$ globally.
\end{proof}

We are now in a position to extend this lemma to points where $u$ is not differentiable
by applying Theorem~\ref{T:second stab}.  This extension is the key
local-implies-global ingredient required for Loeper's argument.
Our proof of Theorem~\ref{T:supporting c-convex} is similar to the proof
of Proposition~2.12 \cite{Loeper07p}.
Under the stronger hypothesis {\bf (A3s)}, Trudinger \& Wang's corrigendum
contains a different approach to a similar result \cite{TrudingerWang08p}
(but with strictly regular costs and a stronger hypothesis on the domains).

\begin{theorem}\label{T:supporting c-convex}
{\bf (Mountains supporting to first-order support globally)}
Use a continuous cost $c:\cl(M \times \bar M) \longrightarrow \R \cup \{+\infty\}$
which has a twisted and weakly regular {\bf (A0)--(A3w)} restriction
$c \in C^4(N)$ to define a pseudo-metric (\ref{metric}) on a horizontally convex
domain $N \subset M \times \bar M$.
If $c$ is unbounded assume that whenever
$(x_k,\bar x_k) \in \dom c$ converges and $c(x_k,\bar x_k) \to +\infty$ then
\begin{align}\label{cost gradient diverges}
\lim_{k \to \infty} |Dc(x_k,\bar x_k)| = +\infty.
\end{align}
Fix $x \in M$ such that 
$\cl \bar M$ appears convex from $x$,
and suppose $\cap_{0\le t \le 1}N(\bar x(t))$ is
dense in $M$ for each geodesic
$t \in [0,1] \longrightarrow (x,\bar x(t)) \in \{x\} \times \cl \bar M$.
Fix $u \in \mathcal{S}^{-c}_{\cl \bar M}$ for which there is a
coordinate neighbourhood
$U \subset M$ of $x$
contained in $\dom \partial^c u$,  and a constant $C<\infty$ such
that each mountain $f:\cl M \longrightarrow \R$ which supports $u$ to first order at
some $y \in U$ satisfies $|\nabla f|\le C$ and $D^2 f \ge -C \I$ distributionally
in the given coordinates on $U$.
If the extended twist hypothesis {\bf (A1$'$)} holds,
any mountain which supports $u$ to first order at $x$
will be dominated by $u$ throughout $\cl M$.
\end{theorem}

\begin{proof}

Take $u \in \mathcal{S}^{-c}_{\cl \bar M}$ and
$x \in U \subset \dom \partial^c u \setminus \partial M$ as in the theorem.
Assume $x \not\in \dom Du$,  since otherwise the conclusion follows immediately
from Lemma \ref{L:supporting c-convex}.
We see $u$ is Lipschitz and semiconvex on $U$
by Lemma \ref{L:inherit},  after noting that
$u$ agrees throughout $U$ with the supremum of those mountains $f \le u$ which support $u$
on $M$ and make contact with $u$ on $U \subset \dom \partial^c u$.
The subdifferential
$\partial u(x) \subset T^*_x M$ is a non-empty convex set from the semiconvexity of $u$;
it is bounded since $u$ is Lipschitz.
Let us first show that if $p^*$ is an extreme point of $\partial u(x)$
then $\cexp_x' p^* \in \partial^c u(x)$.  Extremality implies $x$
is the limit of a sequence $x_k \in U \cap \dom Du$ such that
$p^* = \lim_{k \to \infty} Du(x_k)$, according to \S 25.6 \cite{Rockafellar72}.
Since $\partial^c u(x_k)$ is non-empty,
Lemma \ref{L:supporting c-convex} yields a sequence of mountains
$f_k(\cdot) = -c(\cdot, \bar x_k) + c(x_k,\bar x_k) + u(x_k)$
supporting $u$ 
and focused at points $\bar x_k = \cexp_{x_k}' Du(x_k)$
which converge to a limit $\bar x := \cexp_x' p^*$.
Since $|Dc(x_k,\bar x_k)|$ is bounded by the Lipschitz constant $C$ of $u$ on $U$,
(\ref{cost gradient diverges}) implies $c(x_k,\bar x_k)$ has a finite limit.
It then follows from $u(\cdot) \ge f_k(\cdot)$
that $f(\cdot) = - c(\cdot, \bar x) + c(x, \bar x) + u(x)$ supports $u$ on $\cl M$.

According to Theorem \ref{T:second stab}, the set $\partial^c u(x)$ of foci
corresponding to mountains $h(\cdot) = -c(\cdot, \bar z) + c(x, \bar z) + u(x)$ which
support $u$ at $x$ appears convex from $x$.  Thus there is a geodesic in $N$ connecting
$(x,\bar x)$ to $(x,\bar z)$.  From the twist condition, this means set of
cotangent vectors $q^*$ such that $\cexp_x' q^* \in \partial^c u(x)$
form a convex set in $T^*_x M$.
This convex set contains
$\partial u(x)$,  since it contains the extreme points of $\partial u(x)$.
Finally,  suppose
a mountain $g(\cdot) = -c(\cdot,\bar y) + \mu$ supports $u$ to first order at $x$.
The semiconvexity of $g(\cdot)$ yields some $q^*$ extremal in $\partial g(x)$
with $\bar y = \cexp_x' q^*$.  Moreover,
$q^* \in \partial u(x)$ a fortiori \eqref{first-order support},  hence there is a
mountain $f$ also focused at $\bar y$ which supports $u$ at $x$.
Now $g=f$ since their foci and values at $x$ coincide,  so the theorem is established.
\end{proof}


\begin{remark}[the sphere, the reflector antenna]
Although the hypotheses of Theorem \ref{T:supporting c-convex} appear complicated,
they are obviously satisfied on any product manifold $N = M \times \bar M$ on which
{\bf (A0)-(A3w)} extend to a larger manifold $\tilde N \supset \supset N$
which contains $N$ compactly.
For the distance squared cost $c=d^2/2$ on a
compact and boundaryless Riemannian manifold $(M=\bar M,g)$,
the cost function is globally Lipschitz and semiconvex
\cite{CorderoMcCannSchmuckenschlager01},  hence all mountains
$f(\cdot) = - d^2(\cdot,\bar y)/2 + \mu$ obey the desired local estimates.
Taking $N \subset M \times \bar M$ to be the complement of the cut locus,
 the only delicacies are bi-convexity of
$N$ and the density of $\cap_{0\le t \le 1}N(\bar x(t))$;
on the round sphere these two properties are satisfied, as Loeper knew and
we discuss elsewhere.
On $\Sphere^n \times \Sphere^n$,
with the reflector antenna cost $c(x,\bar x) = -\log |x -\bar x|$ of
Example \ref{E:reflector antenna},
the mountains are smooth away from their foci,  so choosing the
diameter of $U$ small relative to the Lipschitz constant $C$ of $u$,
(\ref{cost gradient diverges}) combines with Proposition \ref{P:semi-convex reflector antenna}
to force all mountains which support $u$ to first order
in $U$ to have foci well outside of $U$.
\end{remark}

\begin{remark}[products ${\mathbf S}^{n_1} \times \cdots \times {\mathbf S}^{n_k} \times \R^l$ and their Riemannian submersions
]
Theorem~\ref{T:supporting c-convex} holds for the distance squared cost on the Riemannian product
$M = \bar M$ of round spheres ${\mathbf S}^{n_1} \times \cdots \times {\mathbf S}^{n_k}\times \R^l$
or its Riemannian submersions. The necessary hypotheses are satisfied according to
Example~\ref{E:new examples} and Remark~\ref{R:new results on products}.
\end{remark}

\section{Estimating the height of the double mountain relative to the sliding mountain
near their point of coincidence}
\label{S:Taylor}

The following local estimate is the main result of this section;
due to Loeper \cite{Loeper07p}, it is crucial to his argument for
H\"older continuity of optimal mappings.
One can view it as a localized version of the double mountain above sliding mountain
principle (Theorem~\ref{T:maximum principle}),
in which strict regularity of the cost is used to
quantify the relative altitude of the double mountain over the sliding
mountain in a neighbourhood of the point where they coincide.
For this entire section,  we shall work in a single Euclidean coordinate
patch; the norm $|p|^2 = p \cdot p$ and derivatives such as $D^2 c = \{c_{ij}\}$
and $D \bar D c = \{c_{i\bar j}\}$ refer to these coordinates.

\begin{proposition}\label{P:A3S and support}
{\bf (Local double-above-sliding-mountain estimate \cite{Loeper07p})}
Let $c \in C^4(\cl N)$ be strictly regular {\bf (A2)--(A3s)} on the closure of
product $N = M \times \bar M$ of two bounded domains $M$, $\bar M$ $\subset {\bf R}^n$,
so that the pseudo-metric (\ref{metric}) obeys
\begin{align}\label{A3 local}
\sec_{(x,\bar x)} (p \oplus \0) \wedge (\0 \oplus \bar p) \ge 2 C_0 |p|^2 |\bar p|^2
\end{align}
for each $(x, \bar x) \in N$ and null vector
$p \oplus \bar p \in T_{(x, \bar x)} N$, 
the norms $|p|$ and $|\bar p|$ being defined using the Euclidean inner product
on the global coordinates. Then there exist positive constants
$r_0 = r_0(\diam \bar M, \|[D\bar Dc]^{-1}\|_{C^0 (N)}, \|c\|_{C^4(N)}, C_0)$
and $C_1 = C_0  \|2 D \bar Dc \|_{C^0 (N)}^{-2} \|[D \bar Dc]^{-1} \|_{C^0 (N)}^{-2}$
such that:
for each coordinate ball $B_r(x) \subset M$ of radius $0<r<r_0$ and geodesic
$t\in[0,1] \longrightarrow (x, \bar x(t)) \in N$ with $\dot {\bar x}(0) \ne \0$,
the function $f(t,y) :=-c(y,\bar x(t)) + c(x,\bar x(t))$ satisfies
$$
\max [ f(0,y), f(1,y) ]
\ge
f(t,y) + C_1 t(1-t) |\bar x (1)-\bar x (0)|^2 |y- x |^2
- \| c \|_{C^3(N) }|y-x |^3
$$
throughout $(t,y) \in [0,1] \times B_r(x)$.
 \end{proposition}

\begin{proof}
Our proof diverges from Loeper's \cite{Loeper07p}, but
has the advantage of producing a radius $r_0>0$ independent of $t \in [0,1]$.

Let $D$ and $D^2$ denote the
gradient and the Hessian operators with respect to the fixed Euclidean coordinates.
For any fixed $w \in T_x M$,  Lemma \ref{L:mixed fourth derivative} asserts
\begin{equation}\label{B sectional curvature}
\frac{d^2}{dt^2} [D^2 f(t,x)] (w, w)
= \frac{1}{2}
\sec_{(x,\bar x (t))} (w \oplus \0) \wedge (\0 \oplus \dot{ \bar x} (t)),
\end{equation}
whence
$$\left|\frac{d^2}{dt^2} [D^2 f(t,x)] (w, w)\right|
\le \frac{1}{2} n^2  |w|^2 |\dot{\bar x}(t)|^2 \sup_{i,\bar j,k,\bar \el} |R_{i\bar j k \bar \el}(x,\bar x(t))|.
$$
from (\ref{sectional curvature}),  and the Cauchy-Schwarz inequality $(\sum_{i=1}^n |w^i|)^2 \le n |w|^2$.

On the other hand, setting $p^*(t) := D f(t,x )=-Dc(x, \bar x (t))$ yields
the linear parameterization
$p^*(t) = p^*(0) + t q^*$ as in (\ref{c-segment}). From
 this, it is easy to see that
$w \oplus \dot{\bar x}(t)$ is an $h$-null vector at $(x, \bar x(t))$
if and only if $q^*_i w^i = 0$ where
\begin{align} \label{def q}
q^*_i = \frac{d p^*_i}{dt}(t) = - c_{i\bar j}(x, \bar x(t)) \dot x^{\bar j}(t).
\end{align}
Note $q^* \ne 0$ because $\dot{\bar x}(0) \ne 0$.
For any vector $w$ in the nullspace of $q^*$,
\eqref{A3 local} and \eqref{B sectional curvature} imply
\begin{align*}
\frac{d^2}{dt^2} [D^2 f(t,x)] (w, w)
 & \ge C_0 |w|^2 |\dot{\bar x}(t)|^2.
\end{align*}

To make these estimates independent of $t$, note \eqref{def q}
combines with nondegeneracy {\bf (A2)} of the cost to imply
that $\dot{\bar x}(t)$ and $q^*$ have comparable Euclidean magnitudes
\begin{align}\label{compare dot x and q}
\frac{1}{\|D\bar D c\|_{C^0 (N)}} |q^*| \le | \dot{\bar x}(t)| \le \|[D \bar D c]^{-1}\|_{C^0(N)} |q^*|.
\end{align}
It follows that
\begin{align}\label{x(1)-x(0)}\nonumber
|\bar x (1)-\bar x(0)|
&\le \int_0^1 |\dot{\bar x}(\tau)| d\tau
\\& \le \|[D \bar D c]^{-1}\|_{C^0(N)} |q^*|.
\end{align}
Combining these estimates, we see
\begin{align}\label{cross-curv and sliding mountain}
\left\{%
 \begin{array}{ll}
\left| \frac{d^2}{dt^2} [D^2 f(t,x)] (w, w) \right|
\le  C_1' |w|^2 |q^*|^2
  & \hbox{ \ \ for general $w$}\\[1ex]
\;\frac{d^2}{dt^2} [D^2 f(t,x)] (w, w)\;
\ge C_0' |w|^2 |q^*|^2
& \hbox{ \ \   for each $w$ with $q^*_i w^i =0$}
  \end{array}%
  \right.
 \end{align}
where
\begin{align}\label{B c0 in proof}
C_0' &=C_0/ \| D \bar D c \|_{C^0 (N)}^2,\\
C_1'
&= \frac{1}{2} n^2 \| [D \bar D c]^{-1}\|_{C^0(N)}^{2} \sup_{i,\bar j,k,\bar \el}\|R_{i\bar j k \bar \el}\|_{C^0(N)}.
\label{B c1 in proof}
\end{align}

In the following, we use \eqref{cross-curv and sliding mountain} and a Taylor expansion argument.
We work in a small ball $B_r (x)\subset M$ of $x$,
whose radius $r$ shall be determined in the course of the proof; note that the condition
$B_r (x) \subset M$ is used only to have the Taylor expansion.
Consider a thin cone
\begin{align*}
K_\theta:=\{ v \in T_x \Omega \mid \Big{|} -\cos\theta \le \frac{q^*_i v^i}{|q^*||v|} \le \cos\theta \}
\end{align*}
around the nullspace of $q^*$.
Here $\pi/4<\theta<\pi/2$ will subsequently be chosen to be large, depending on $C_0'$ and $C_1'$.
We shall establish the theorem by taking $\theta \approx \pi /2$ in case 1 and
$r>0$ small in Case 2 below.

\subsubsection*{Case 1 (Second order):} Let $v \in  B_r(\0)\cap K_\theta$.
Then
\begin{align*}
&\max [ f(0,x+v), f(1,x+v) ] - f(t,x+v) \\
&\ge \ t f(1,x+v) + (1-t) f(0,x+v) - f(t,x+v) \\
&\ge [t D^2 f(1,x) + (1-t) D^2 f(0,x)-D^2 f(t,x) ] (v,v)
- \| c \|_{C^3(N)} |v|^3
\end{align*}
where the first order term $v(t p^*(0) + (1-t) p^*(1) - p^*(t))$ of the Taylor
expansion has vanished due to the linearity of $p^*(t) = p^*(0) + t q^*$.
To bound the second term using \eqref{cross-curv and sliding mountain},
rotate the Euclidean coordinates so that $v=(v_1, V)$,  where
$v_1 = q^*_iv^i/|q^*|$ and $(0,V)$ lies in the nullspace $q^*$.
Note that for $v \in B_r(\0) \cap K_\theta$ we have
$|v_1| < |v|\cos \theta $ and hence $|v|^2\sin^2 \theta \le |V|^2$.
After the mean value theorem, \eqref{cross-curv and sliding mountain} yields
\begin{align*}
&[t D^2 f(1,x) + (1-t) D^2 f(0,x)-D^2 f(t,x) ] (v, v)
\\
&\ge  \frac{1}{2}C_0' t(1-t) |q^*|^2  |V|^2
- \frac{1}{2}C_1'  t(1-t) |q^*|^2  (2|v_1||V|+ |v_1|^2)
\\&\ge \frac{1}{2}\Big{[}C_0' \sin^2 \theta  -3 C_1' \cos \theta \Big{]} t(1-t) |q^*|^2 |v|^2\\
& \ge \frac{1}{4}C_0' t(1-t) \left| \frac{\bar x(1) - \bar x(0)}{\|[D\bar D c]^{-1}\|_{C^0(N)}}\right|^2 |v|^2
\end{align*}
noting \eqref{x(1)-x(0)}, and provided we choose $\pi/4<\theta <\pi/2$ sufficiently large that 
\begin{align}\label{alpha}
\sin^2 \theta \ge \frac{1}{2} + \frac{ 3C_1'}{C_0'} \cos \theta.
\end{align}
This establishes the desired inequality for Case 1 for $\theta$ close to $\pi/2$.

\subsubsection*{Case 2 (First order):} Let $v \in B_r (\0) \setminus K_\theta$,
and recall
\begin{align*}
p^*(1) - p^*(t) &=
(1-t) q^* \\
p^*(0) - p^*(t) & = -t q^*.
\end{align*}
Now either $q^*_i v^i \le -|q^*| |v| \cos\theta$ or $|q^*| |v| \cos\theta \le q^*_i v^i$.
If $q^*_i v^i \le 0$, Taylor expansion yields
\begin{align*}
& \max [ f(0,x+v), f(1,x+v) ] -f(t,x+v)\\
   & \ge    f(0,x+v) - f(t,x+v) \\
    & = (p^*_i(0) - p^*_i(t) ) v^i
    + (D^2 f(0,x) - D^2 f(t,x) )(v,v)  + O (|v|^3) \\
    &  \ge t \cos \theta |v| |q^*|
    - t\| D^2 \bar Dc\|_{C^0 (N)}  \max_{0 < \tau < t} \{|\dot{\bar x}(\tau) |\} |v|^2 - \| c\|_{C^3(N)} |v|^3
\\&  \ge t |v| |q^*| \Big(\cos \theta
   - \| c\|_{C^3 (N)} \|[D\bar D c]^{-1}\|_{C^0(N)} |v| \Big) - \| c\|_{C^3(N)} |v|^3
\\&  \ge t |v| |\bar x(1)-\bar x(0)| \Big(\|[D\bar D c]^{-1}\|_{C^0(N)}^{-1}\cos \theta
   - \| c\|_{C^3 (N)}  |v| \Big) - \| c\|_{C^3(N)} |v|^3
\end{align*}
where \eqref{compare dot x and q}--\eqref{x(1)-x(0)} have been used.
Similarly,  if $q^*_i w^i \ge 0$, then 
\begin{align*}
& \max [ f(0,x+v), f(1,x+v) ]-f(t,x+v)\\
& \ge (1-t) |v| |\bar x(1) -\bar x(0)|
\Big(\|[D\bar D c]^{-1}\|_{C^0(N)}^{-1}\cos \theta - \| c\|_{C^3 (N)}  |v| \Big)
 -\|c\|_{C^3(N)}|v|^3.
\end{align*}
If we assume
\begin{align}\label{r2}
 & |v|< r \leq r_0 := \frac{\| [D\bar D c]^{-1}\|_{C^0 (N)}^{-1} \cos\theta}{
 \| c\|_{C^3 (N)} + C_1 \diam \bar M },
\end{align}
then for $v \in B_r(\0) \setminus K_\theta$ with either sign of $q^*_i v^i$ we find
\begin{align*}
&\max [ f(0,x+v), f(1,x+v) ] - f(t,x+v)
\\\nonumber
& \ge  C_1 t(1-t)|v|^2 |\bar x(1) - \bar x(0)|^2
-\|c\|_{C^3(N)}|v|^3.
\end{align*}

Combining Case 1 with Case 2 establishes the proposition; the constants have the desired dependence
according to \eqref{B c0 in proof},
\eqref{B c1 in proof}, 
\eqref{alpha}, \eqref{r2}, and \eqref{Riemann curvature}.
%
\end{proof}

\section{On fitting sausages into balls}
\label{S:sausage}

Loeper's H\"older continuity argument for optimal mappings $F:\Omega \longrightarrow \bar \Omega$
between domains $M = \Omega \subset \R^n$ and $\bar M = \bar \Omega \subset \R^n$
is based on a localized volume comparison.  To set this up,
for any optimal map which fails to be $C^{1/5}(\Omega;\cl \bar \Omega)$, the next proposition
identifies $\epsilon, \delta>0$ and a null geodesic
$t \in [0,1] \longrightarrow (x,\bar x(t)) \in \Omega \times \bar \Omega$,
such that an $\epsilon$-ball in coordinates around $x$
and a $\delta$-sausage $\mathcal{N}_\delta$ around the middle third of $\bar x(t)$ satisfy
$F^{-1}(\mathcal{N}_\delta) \subset B_\epsilon(x)$.  More precisely, we establish
Proposition \ref{P:sausage construction} for
the potential function $u$ of the mapping $F = \cexp \circ Du$.

We start by defining $\delta$-neighbourhoods and relative $\delta$-neighbourhoods.
As usual, $[\bar x,\bar y]$ denotes the Euclidean convex hull of the two points
$\bar x,\bar y \in \Rn$, and we write 
$Y \subset\subset \bar \Omega$
to indicate that the closure of $Y$ is a compact subset of $\bar \Omega$.

\begin{definition}{\bf (Relative $\delta$-neighbourhoods)}
Given $Y \subset \bar \Omega \subset  \R^n$, denote
\begin{align*}
\mathcal{N}_\delta (Y) = \{ \bar x \in \R^n &\mid
\hbox{there exists $\bar y \in Y$ such that $|\bar x- \bar y| < \delta$} \},\\
\mathcal{N}^{\bar \Omega}_\delta (Y) = \{ x \in \bar \Omega \ \ &\mid
\hbox{there exists $\bar y \in Y$ such that $|\bar x - \bar y| < \delta$}  \\
& \qquad \hbox{and the line segment $[\bar x, \bar y] \subset \bar \Omega$}
\}.
\end{align*}
\end{definition}

\noindent
Note $\mathcal{N}^{\bar \Omega}_\delta(Y) = \bar \Omega \cap \mathcal{N}_\delta(Y)$
if $\bar \Omega \subset \Rn$ is convex.
The point of this section is the following
version of Loeper's Proposition 5.3 \cite {Loeper07p}.
\begin{proposition}{\bf (Fitting sausages into balls)}
\label{P:sausage construction}
Let $c \in C^4(\cl N)$ be strictly regular {\bf (A2)--(A3s)} on the closure of the product
$N=\Omega \times \bar \Omega$ of two bounded domains $\Omega$ and
$\bar \Omega \subset \Rn$.  
Fix $u:\Omega \longrightarrow \R$ semiconvex with constant $C$,  such that
any mountain focused in $\cl \bar \Omega$ and supporting $u$ to first-order 
is dominated by $u$ throughout $\Omega$.
Choose a pair of points $(x_0, \bar x_0)$ and $(x_1,\bar x_1)$
from $\partial^c u \subset \cl(\Omega \times \bar \Omega)$
with $\bar x_0 \ne \bar x_1$ and $[x_0,x_1] \subset \Omega$. From
the definition of $\partial^c u = \partial^c_{\cl \Omega} u$, both mountains
\begin{align}\label{C two mountains}
f_i (\cdot) = - c( \cdot , \bar x_i ) + c (x_i , \bar x_i) + u(x_i), \qquad i=0,1,
\end{align}
support $u$,  hence there exists a point $x$ where the valley of indifference
$f_0 (x) = f_1 (x)$ intersects the line segment $[x_0, x_1] \subset \Omega$.
Suppose a geodesic $t \in ]0,1[ \longrightarrow (x, \bar x(t)) \in N$
for the pseudo-metric \eqref{metric} links $\bar x_0 = \lim_{t \to 0} \bar x(t)$
to $\bar x_1 = \lim_{t \to 1} \bar x(t)$.
Then taking $C', C'' , C'''>0$ sufficiently small
depending only on $r_0$ and $C_1$ from Proposition~\ref{P:A3S and support},
the semiconvexity constant $C$, 
$\|c\|_{C^3(N)}$, 
and $\diam\bar\Omega$, implies: if
$ 
 | x_1 - x_0 |  < (C'C'')^2 |\bar x_0 - \bar x_1 |^5
$ 
then for any $\epsilon>0$ in the non-empty interval
\begin{align}\label{epsilon}
   \frac{1}{C'}  \sqrt{\frac{|x_0 - x_1 |}{|
\bar x_0 - \bar x_1 |}} \le &\epsilon \le C'' |\bar x_0 - \bar x_1|^2, \qquad {\rm and}
\\\label{delta}
& \delta := C''' \, \epsilon \, |\bar x_0 - \bar x_1 |^2,
\end{align}
such that $B_\epsilon(x) \subset \subset \Omega$ we have
$[\cl B_\epsilon(x) \times \{\bar z\}] \cap \partial^c u$ non-empty for each
$\bar z \in \mathcal{N}^{\bar \Omega}_\delta (\{ \bar x(t) \mid
\textstyle \frac{1}{3} \le t \le \frac{2}{3}\})$.
\end{proposition}

\begin{proof}
We present a modified version of the original proof of Loeper \cite{Loeper07p}.

In the course of the proof, $\delta >0$  will be determined.
Given
$\bar z \in \mathcal{N}^{\bar \Omega}_\delta (\{ \bar x(t) , t \in [1/3, 2/3]\})$,
define the mountain
\begin{align*}
g_{\bar z} ( \cdot ) := -c(\cdot, \bar z) + c(x , \bar z) + u(x)
\end{align*}
focused at $\bar z$ which agrees with $u(\cdot)$ at $x$ but need not
be a supporting mountain.
Our goal is to specify $\epsilon, \delta >0$ with
$B_\epsilon (x) \subset\subset \Omega$ such that
\begin{align}\label{C boundary estimate}
u(\cdot) \ge  g_{\bar z}(\cdot)  \quad \hbox{throughout $\partial B_\epsilon (x)$}.
\end{align}
Once this is established, it implies the existence of $z$ which
realizes the global minimum of $u(\cdot)  - g_{\bar z}(\cdot)$ on $\cl B_\epsilon(x)$.
Adjusting the height of $f_{\bar z}(\cdot)$ by $u(z) - g_{\bar z}(z)$ yields a
mountain which supports $u(\cdot)$ to first-order at $z$,  hence globally by hypothesis.
This implies $(z,\bar z) \in \partial^c u$ as desired.

By construction, $u \ge \max[f_0, f_1]$ on $\Omega$.
The choice of $\bar z$ implies $t \in [1/3, 2/3]$ exists with
$|\bar z  - \bar x(t) | \le \delta$ and
\begin{align}\label{segment in bar Omega}
\hbox{the line segment} \ [\bar z, \bar x(t)] \subset \bar \Omega.
\end{align}
Assume $f_0(x) =0 = f_1(x)$ without loss of generality, so
\begin{align*}
f_t (\cdot) := -c (\cdot, \bar x (t) ) + c(x, \bar x (t)), \qquad t\in [0,1],
\end{align*}
extends \eqref{C two mountains}.
Let $\epsilon < r_0$ be chosen to satisfy $B_\epsilon (x) \subset \subset \Omega$ later,
where $r_0$ is from Proposition~\ref{P:A3S and support}.
Then for each
$y \in \partial B_\epsilon (x)$,
\begin{align}\label{phi f bar z}\nonumber
u(y) - g_{\bar z} (y)
&\ge \max[f_0 (y) , f_1 (y)] -g_{\bar z} (y) \\
&\ge f_t (y) - g_{\bar z} (y)  + C_1 t(1-t) |\bar x_0 -\bar x_1 |^2 | y-x|^2
- \|c(\cdot, \cdot)\|_{C^3(N)} |y-x|^3.
\end{align}
To achieve \eqref{C boundary estimate}
we want the right hand side nonnegative. First note
\begin{align}\label{f t f bar z}\nonumber
f_t (y) - g_{\bar z} (y)
&= -c(y, \bar x (t)) + c(x , \bar x(t)) + c(y, \bar z) - c(x, \bar z) - u (x) \\
&\ge - \|c\|_{C^2(N)}|y-x | | \bar z-\bar x (t)| - u(x)
\end{align}
by second derivative estimates; this estimate is the only place where we use the condition
\eqref{segment in bar Omega} in the proof.
Recall $x= (1-s) x_0 + s x_1$ for some $s \in [0,1]$. Semi-convexity implies
\begin{align*}
u(x) &\le (1-s)u(x_0 )  + su(x_1) + C |x_0 -x_1 |^2/2
\end{align*}
where
\begin{align*}
& (1-s) u(x_0 ) + s u(x_1)\\
&=(1-s)f_0 (x_0 ) + sf_1 (x_1 )\\
&\le (1-s) D f_0 (x ) (x_0-x ) + sD f_1 (x ) (x_1 -x) + \frac{1}{2}\|c\|_{C^2(N)} (|x_0 - x  |^2 + | x_1 - x |^2)\\
&\le s(1-s) [D f_0 (x )-D f_1(x)](x_0-x_1) + \|c\|_{C^2(N)}  |x_0 - x_1  |^2/2\\
&\le s(1-s)\|D\bar D c\|_{C^0 (N)} |\bar x_0 - \bar x_1| |x_0 - x_1 | + \|c\|_{C^2(N)}  |x_1 - x_0 |^2/2
\end{align*}
Therefore,
\begin{align*}
u(x) &\le \|D\bar D c\|_{C^0 (N)} |\bar x_0 - \bar x_1| |x_0 - x_1 | +
(\|c\|_{C^2(N)} +C) |x_0 -x_1 |^2/2.
\end{align*} From
\eqref{phi f bar z} and \eqref{f t f bar z}, the desired inequality
\eqref{C boundary estimate} reduces to the following:
\begin{align*}
& 0 \le \ C_1 t(1-t) |\bar x_0 -\bar x_1 |^2 \epsilon^2 - \|c\|_{C^2(N)}\epsilon \delta -
\|c\|_{C^2(N)}|\bar x_0 -\bar x_1 ||x_0 - x_1 |\\
& \ \ \ \ \  -(\|c\|_{C^2(N)}+ C) |x_1 - x_0 |^2/2
-\|c\|_{C^3(N)} \epsilon^3
\end{align*}
which can be satisfied if the negative terms are small relative to the positive term
$C_1 t(1-t) |\bar x_0 -\bar x_1 |^2 \epsilon^2 \gtrsim |\bar x_0 -\bar x_1 |^2 \epsilon^2.$
This will be the case if
\begin{align*}
\delta & \lesssim |\bar x_0 - \bar x_1 |^2 \epsilon,\\
|x_0 - x_1 | & \lesssim |\bar x_0 -\bar x_1 | \epsilon^2,\\
\epsilon & \lesssim 1,\\
\epsilon  & \lesssim |\bar x_0 - \bar x_1 |^2
\end{align*}
where all the relevant constants (for $\lesssim$)
depend only on $C$, $C_1$, $\|c\|_{C^2(N)}$, and $\|c\|_{C^3(N)}$.
Fixing suitable constants $C'$, $C''$ and $C'''$ implies
the choices \eqref{epsilon} and \eqref{delta} suffice,
provided $C''$ is taken small enough
depending on $r_0$ 
and $\diam \bar \Omega$ to have $\epsilon < r_0$ 
as required.  For the interval \eqref{epsilon} to be non-empty,  we need
$
|x_0 - x_1 | < (C'C'')^2 |\bar x_0 - \bar x_1 |^5,
$ 
which is enough to ensure compatibility of the preceding four inequalities.
\end{proof}

\section{H\"older continuity of optimal transport by volume comparison}
\label{S:volume comparison}


Let us now show how the H\"older continuity of optimal mappings ---
Loeper's Theorem 3.4 and Proposition 5.5 \cite{Loeper07p} ---
follows from Theorem~\ref{T:supporting c-convex}, 
Proposition~\ref{P:A3S and support},
Proposition~\ref{P:sausage construction},
and Loeper's volume comparison argument, which we now recall.
Note that H\"older continuity of the optimal transport map $F = \cexp \circ Du$
implies H\"older differentiability of the corresponding potential
$u \in \mathcal{S}^{-c}_{\cl \bar M}$.

\begin{theorem}\label{T:Hoelder}{\bf (H\"older continuity \cite{Loeper07p})}
Assume $c \in C^4(\cl N)$ is bi-twisted and strictly regular {\bf (A1)--(A3s)} on $\cl N$, where
$N=\Omega \times \bar \Omega \subset \Rn \times \Rn$ 
is a bounded domain
bi-convex with respect to the pseudo-metric \eqref{metric}.
Fix $m>0$, and let $\rho$ and $\bar \rho$ be probability measures on $\Omega$ and
$\bar \Omega$ with
Lebesgue densities
$d\bar \rho/d\vol \ge m$ throughout $\bar \Omega$ and $d\rho/d\vol \in L^\infty(\Omega)$.
Then there exists a map $F \in C^{1/\max\{5,4n-1\}}_{loc}(\Omega,\cl \bar \Omega)$
between $\rho$ and $\bar \rho$ which is optimal with respect to the transportation cost $c$.
\end{theorem}

\begin{proof}

Note $c \in C^4(\cl N)$ since strict regularity extends to a neighbourhood of $\cl N$
by hypothesis.
By Kantorovich duality,
there exists $u \in \mathcal{S}^{-c}_{\cl \bar \Omega}$
such that the optimal $\gamma$ in (\ref{Kantorovich})
vanishes outside $\partial^c u = \partial^c_{\cl \Omega} u$
\cite{RachevRuschendorf98} \cite{Villani03}.
Moreover, both $u$ and its transform $\bar u := u^{c^*}_{\cl \Omega}$ are globally
Lipschitz and semiconvex from Lemma \ref{L:inherit} and our hypotheses on the cost.
By Corollary \ref{C:Lipschitz/semiconvex}, $\cl \Omega \subset \dom \partial^c u$.
Since $u(x) + \bar u(\bar x) + c(x,\bar x) \ge 0$ with equality on $\partial^c u$,
the bi-twist condition yields
$\partial^c u \cap (\dom Du \times \dom \bar D \bar u) \subset
\graph (F) \cap \antigraph(G)$ where $F := \cexp \circ Du$ and $G := \cdexp \circ D\bar u$
 \cite{Gangbo95} \cite{Carlier03} \cite{MaTrudingerWang05}.
Here $\antigraph(G) := \{(G(\bar x),\bar x) \mid \bar x \in \dom D\bar u\}$,
and $\dom Du \times \dom \bar D \bar u \subset \Omega \times \bar \Omega$
is a set of full mass for $\rho \otimes \bar \rho$, a fortiori for $\gamma$.
Moreover, $x=G(F(x))$ holds for $\rho$-a.e. $x \in \Omega$,
and $\bar x = F(G(\bar x))$ for $\vol$-a.e.\ $\bar x \in \bar \Omega$.
We shall establish the theorem following \cite{Loeper07p}.

Since $N$ is horizontally-convex and {\bf (A1$'$)} is satisfied,
Theorem~\ref{T:supporting c-convex} ensures the function $u$
satisfies the hypotheses of Proposition~\ref{P:sausage construction}:
namely that that any mountain supporting $u$ to first-order supports $u$ globally.
Fix $\eta>0$, and let $\Omega_\eta := \Omega \setminus \mathcal{N}_{\eta}(\partial \Omega)$.
Then $x_0,x_1 \in \Omega_{\eta}$ with $|x_0 - x_1| < \eta/2$ implies
$[x_0,x_1]\subset \Omega_{\eta/2}$,  so any $\epsilon<\eta/2$
implies $\mathcal{N}_\epsilon([x_0,x_1]) \subset\subset \Omega$.

Pick distinct points $x_0 , x_1 \in \Omega_\eta$ with 
$(x_i,\bar x_i) \in  \partial^c u(x_i) \cap (\dom Du \times \dom \bar D\bar u)$, $i=0,1$,
such that
\begin{align}\label{D good points}
|x_0 - x_1 | & < \min\{
(C'C'')^2 |\bar x_0 - \bar x_1 |^5
,(\frac{\eta C'}{2})^2|\bar x_0 -\bar x_1|
,\frac{\eta}{2}
\}.
\end{align}
If there no such pair of points exists, then
$F \in {C^{1/5}(\Omega_\eta; \cl\bar \Omega)}$
with a H\"older constant depending on $\eta,C',C'',$ and the Euclidean path
diameter of $\Omega_\eta$.
Choosing $x$ as in Proposition~\ref{P:sausage construction},
the vertical-convexity of $N$
yields a geodesic $t \in [0,1] \to (x, \bar x(t)) \in N$
connecting $(x,\bar x_0)$ to $(x,\bar x_1)$.

Unless $(x_0,\bar x_0)$ and $(x_1,\bar x_1)$ can be chosen to make
$\epsilon = (|x_0-x_1|/|\bar x_0 -\bar x_1|)^{1/2}/C'$ and hence
$\delta = C'''\epsilon|\bar x_0 -\bar x_1|^2$
arbitrarily small in \eqref{epsilon}--\eqref{delta},  the map $F$ is Lipschitz
on $\Omega_\eta$. Noting
$0< \epsilon<\eta/2$, Proposition \ref{P:sausage construction} yields
\begin{align}\label{sausage conclusion}
& \mathcal{N}^{\bar \Omega}_\delta (\{ \bar x(t), t \in [1/3, 2/3]\}) \subset \partial^c u(B_\epsilon (x)).
\end{align}
Since the cost $c$ is bi-twisted, 
the map $G := c^*\Exp \circ D \bar u$ is well-defined 
on the set $\dom D\bar u \subset \bar \Omega$ of
differentiability for $\bar u := u^{c^*}_{\cl \Omega}$. 
This set has full Lebesgue measure (its complement has Hausdorff dimension
$n-1$),  and if $(x,\bar x) \in \partial^c u$ with $\bar x \in \dom D\bar u$
then $x=G(\bar x)$.
Now (\ref{sausage conclusion}) asserts
\begin{align}\label{sausage in the image of meat ball}
B_\epsilon (x)
&\supset G \Big( \mathcal{N}^{\bar \Omega}_\delta (\{ \bar x(t), t \in [1/3, 2/3]\})
\cap \dom D\bar u \Big) \cap \dom Du
\\ \nonumber
&= F^{-1} \Big( \mathcal{N}^{\bar \Omega}_\delta (\{ \bar x(t), t \in [1/3, 2/3]\})
\cap \dom D\bar u \Big) \cap \dom Du.
\end{align}
Since $F_\# \rho = \bar \rho$ has density $d\bar \rho / d\vol \ge m$,
integrating $\rho$ over these two sets yields
\begin{align}\label{integrated meatball}
\|\rho\|_{L^\infty(\Omega)} \vol B_\epsilon(x) \ge
m \vol (\mathcal{N}^{\bar \Omega}_\delta (\{ \bar x(t), t \in [1/3, 2/3]\})).
\end{align}
%

For $c\in C^4(N)$ and twisted, 
$\cexp:Dc(x,\bar \Omega) \longrightarrow \bar \Omega$
gives $C^3$ diffeomorphism.
The pre-image of $\{\bar x(t)\}_{t \in [1/3,2/3]}$ under this map is a segment
in $Dc(x,\bar \Omega) \subset T^* \Omega$ according to Lemma \ref{L:c-segments are geodesics}.
Noting that bi-convexity of $N$ implies convexity of $Dc(x,\bar \Omega)$,
it is not hard to see
\begin{align}\label{CN and N}
\vol (\mathcal{N}^{\bar \Omega}_\delta (\{ \bar x(t), t \in [1/3, 2/3]\}))
\gtrsim \vol(\mathcal{N}_\delta (\{ \bar x(t), t \in [1/3, 2/3]\}))
\end{align}
for $\delta> 0$ sufficiently small depending on $c$ and $N$ but not on
$\bar x_0,\bar x_1$ or $x$. Recalling (\ref{delta}) then yields
\begin{align}\label{volume compare}
\epsilon^n \|\rho\|_{L^\infty(\Omega)}
&\gtrsim m \delta^{n-1} |\bar x_0 - \bar x_1| 
\\&= m (C'''\epsilon)^{n-1} |\bar x_0 - \bar x_1|^{2n-1}.
\nonumber
\end{align}
Thus whenever $\epsilon = (|x_0-x_1|/|\bar x_0 -\bar x_1|)^{1/2}/C'$ and hence $\delta$
is sufficiently small we find
$$
\|\rho\|_{L^\infty(\Omega)}^2 |x_0-x_1| \gtrsim (m C')^2(C''')^{2n-2} |\bar x_0 - \bar x_1|^{4n-1}
$$
so $F \in C^{1/(4n-1)}\cup C^{0,1}(\Omega_\eta, \cl \bar \Omega)$.
Since $\eta>0$ was arbitrary,  the proof is complete.
\end{proof}


\begin{remark}[Rougher measures]\label{R:rougher measures}
Note that no smoothness is required of the measures in the
preceding theorem:  $\bar \rho$ need not even be absolutely continuous with
respect to Lebesgue and the support of $\rho$ need not be connected.
Replacing the hypothesis $d\rho/d\vol \in L^\infty(\Omega)$ by
the existence of $p \in ] n, +\infty]$ and $C(\rho) > 0$ such that
\begin{align}\label{p-condition}
&\rho (B_\epsilon (x) ) \le C(\rho) \epsilon^{n-{n}/{p}}
\qquad \hbox{for all $\epsilon \ge 0 ,\quad x \in \Omega$},
\end{align}
Loeper  \cite{Loeper07p} was able to conclude
$F \in C^{1/5}_{loc} \cup C^{(1-\frac{n}{p})/(4n-1-\frac{n}{p})}_{loc}(\Omega, \cl \bar \Omega)$
by assuming $n \ge 2$ and modifying the passage from
\eqref{sausage in the image of meat ball} to \eqref{integrated meatball}
in the obvious way in the argument above.
Replacing \eqref{p-condition} by the even weaker condition that there exists
$f:\R_+ \longrightarrow \R_+$ with $\lim_{\epsilon \to 0} f(\epsilon) = 0$ such that
\begin{align}\label{p=n condition}
&\rho (B_\epsilon (x) ) \le f(\epsilon)\epsilon^{n-1}
\qquad \hbox{for all $\epsilon \ge 0,\quad x \in \Omega$},
\end{align}
Loeper was able to conclude continuity of $F:\Omega \longrightarrow \cl \bar \Omega$ similarly.
In these cases $\rho$ need not be absolutely continuous with respect to Lebesgue measure,
though \eqref{p=n condition} precludes it from
charging rectifiable sets of dimension less than or equal to $n-1$.
\end{remark}

\begin{remark}[H\"older continuity of the boundary map]
Strictly regular and bi-twisted on $\cl N$ means
$N = \Omega \times \bar \Omega$ is compactly contained in a larger domain
$N' \subset \Rn \times \Rn$ on which $c \in C^4(N')$ remains
strictly regular and bi-twisted 
{\bf (A0)--(A3s)}.  If we allow $N' \not\subset \Rn \times \Rn$ to be a manifold
which contains $\cl(N)$ compactly,  we see the boundedness hypotheses on $M=\Omega$
and $\bar M = \bar \Omega$ become unnecessary in Propositions
\ref{P:A3S and support}--\ref{P:sausage construction} and Theorem \ref{T:Hoelder}.
If a bi-convex product domain
$N'=\Omega' \times \bar \Omega \subset \subset \tilde N$
exists with $\Omega \subset\subset \Omega'$, then for $n \ge 2$
Theorem \ref{T:Hoelder} asserts $F \in C^{1/(4n-1)}_{loc}(\Omega',\cl \bar \Omega)$,
hence $F \in C^{1/(4n-1)}(\Omega,\cl \bar \Omega)$.
\end{remark}

\section{Continuity of optimal transport on the sphere}
\label{S:spherical continuity}

This section presents
a continuity result of optimal transport on the sphere as an \emph{a priori} step
toward H\"older continuity.
%
%
%
It is known from the work of Cordero-Erausquin, McCann \& Schmuckenschl\"ager
\cite{CorderoMcCannSchmuckenschlager01} that
on a compact $n$-dimensional Riemannian manifold $(M=\bar M, g)$
with the Riemannian distance squared cost $c=d^2/2$ of Example \ref{E:Riemannian},
the optimal map stays away from cut locus for $\rho$-almost all points,
when the measures $\rho$ and $\bar \rho$ do not charge rectifiable sets of Hausdorff dimension
$n-1$.  On the round sphere,  Delano\"e \& Loeper provided a quantitative
estimate of the distance separating $F(x)$ from the cut locus of $x$ in terms of
$\|\log d\rho/d\vol\|_{L^\infty(\Sphere^n)}$ and
$\|\log d \bar \rho/ d\vol\|_{L^\infty(\Sphere^n)}$.
In Corollary~\ref{C:UAC} we use the preceding ideas to recover
a form of this statement which, though less sharp than the result of
\cite{DelanoeLoeper06},  still suffices for our purpose: namely, to
give a self-contained proof of Loeper's H\"older continuity result on the sphere,
Theorem \ref{T:regular sphere}.
It is not known whether optimal maps with respect to strictly regular costs
$c=d^2/2$ on other Riemannian manifolds stay uniformly away from the cut locus,
however.

\begin{theorem}\label{T:C 0 sphere}{\bf (Continuity of maps on the sphere)}
Let $M=\partial \Ball_1(\0) \subset \R^{n+1}$ be the Euclidean unit sphere,
also denoted $\Sphere^n$, equipped with either cost:
\begin{itemize}
\item[(a)] $c(x,\bar x) = \dist^2(x , \bar x)/2$, where $\dist$
denotes the great circle distance in $\Sphere^n$,
\item[(b)] $c(x, \bar x) = -\log |x -\bar x|$ where $|\cdot|$ denotes the Euclidean distance of $\R^{n+1}$.
\end{itemize}
Fix probability measures $\rho, \bar \rho$ on the round sphere $\Sphere^n$,
one satisfying \eqref{p=n condition} and
the other having a density $d \bar \rho / d\vol \ge m>0$ 
with respect to Riemannian volume 
on $\Sphere^n$.
Then some continuous map $F$ on $\Sphere^n$ between $\rho$ and
$\bar \rho$ is optimal with respect to the cost function $c$ and
takes the form $F = \cexp \circ Du$ with $u \in C^1(\Sphere^n) \cap \mathcal{S}^{-c}_{\Sphere^n}$.
\end{theorem}

\begin{proof}
Let $N \subset \Sphere^n \times \Sphere^n$ denote the set of points where $c$ is smooth,
and $\Delta = \Sphere^n \times \Sphere^n \setminus N$ its complement, the singular set:
\begin{itemize}
\item for cost (a), $\Delta= \{ (x, -x) \mid \hbox{$-x$ is the antipode of $x \in \Sphere^n$} \}$
\item for cost (b), $\Delta= \{ (x, x) \mid x \in \Sphere^n \}$.
\end{itemize}
Note the pseudo-Riemannian metric $h$ of \eqref{metric} is strictly regular on
$N$, and $N$ is bi-convex with respect to $h$. Moreover, this bi-convexity is $2$-uniform
(a fortiori strict) in the sense that $Dc(x,\bar N(x)) \subset T_x^* \Sphere^n$ is either
(a) a disc of radius $\pi = \diam \Sphere^n$ or (b) the entire cotangent space.

The form $F = \cexp \circ Du$ of the optimal map and Lipschitz potential
$u \in \mathcal{S}^{-c}_{\Sphere^n}$ can be found in (a) McCann \cite{McCann01},
where $u$ is semiconvex according to Cordero-Erausquin, McCann \& Schmuckenschl\"ager
\cite{CorderoMcCannSchmuckenschlager01}.  The same form can be deduced
for cost (b) analogously,  using the Lipschitz and semiconvexity from
Proposition \ref{P:semi-convex reflector antenna}.
To derive a contradiction,
suppose $u$ is not differentiable at some point $x \in \Sphere^n$.
Since $u$ is globally Lipschitz and semiconvex
this means the compact convex set $\partial u (x) \subset T^*_x \Sphere^n$
consists of more than one point.  Choose two distinct co-vectors
$p^* \ne q^*$ extremal in $\partial u (x)$.
Extremality yields sequences $x_k \to x$ and $y_k \to x$ in $\dom Du$
such that $Du(x_k) \to p^*$ and $Du(y_k) \to q^*$ by \S 25.6 \cite{Rockafellar72}.
Lemma \ref{L:supporting c-convex} asserts $(x_k,Du(x_k)) \in \cl \dom \cexp$
and provides a
supporting mountain $-c(\cdot, \bar x_k)+c(x_k,\bar x_k) + u(x_k)$ 
for $u$ at $x_k$ which is focused at $\bar x_k = \cexp_{x_k}' Du(x_k)$.
The limit $k \to \infty$ yields a mountain supporting $u$ at $x$ and focused
at $\bar x = \cexp_x' p^*$; for the same reason, another mountain supporting $u$ at $x$
is focused at $\bar y=\cexp_x' q^*$.
%
%
Whether or not these mountains are distinct,  a geodesic
$t' \in ]0,1[ \longrightarrow (x,\bar z(t'))$ in $\{x\} \times \bar N(x)$
which extends from $(x,\bar x)$ to $(x,\bar y)$
is defined by $\bar z(t') := \cexp_x[(1-t')p^* + t' q^*]$ according to
Lemma \ref{L:c-segments are geodesics}.
Applying Theorem~\ref{T:second stab} to (a) Example~\ref{E:sphere} and
(b) Example~\ref{E:reflector antenna} yields
$\bar z(t') \in  \partial^c u(x)$ for all $ t'\in [0,1]$.

To exploit Proposition~\ref{P:sausage construction},
first localize the domains to subdomains of $\R^n$ in the following way. From
the strict vertical convexity of $\Sphere^n\times \Sphere^n \setminus \Delta$,
there is a subdomain $\bar \Omega \subset\subset \bar N (x)$
containing $\{\bar z (t')\}_{t' \in [1/4, 3/4]}$.
Reparameterize this portion of the curve to get a geodesic
$t \in [0,1] \longrightarrow (x,\bar x(t))$ with distinct endpoints
$\bar x (0) = \bar z (1/4)$ and $\bar x(1) = \bar z (3/4)$.
Choose a coordinate neighborhood $\Omega \subset \Sphere^n$ of $x$ which corresponds
to a bounded domain $\Omega \subset \subset \R^n$; $\bar \Omega$ can be mapped to
bounded domain in $\R^n$ by using its diffeomorphic preimage under $\cexp_x$.

The rest of proof is an easy application of Proposition~\ref{P:sausage construction}
to this case $x_0 = x = x_1$ and $\bar x_0 = \bar x(0) \ne \bar x_1 = \bar x(1)$;
we employ the notation of that proposition.
Note that $u$ satisfies the hypothesis of Proposition~\ref{P:sausage construction}
according to 
Theorem~\ref{T:supporting c-convex}.

We apply Proposition~\ref{P:sausage construction} to
$\graph F \subset \partial^c u$.
First notice from \eqref{epsilon} that $\epsilon$ may
be chosen arbitrarily in the range $]0, C'' | \bar x_0 - \bar x_1|^2]$.
To compare the
volume of the ball $B_\epsilon(x)$ to the sausage
$\mathcal{N}^{\bar \Omega}_\delta (\{ \bar x(t), t \in [1/3, 2/3]\})$,
we repeat steps \eqref{sausage in the image of meat ball}--\eqref{volume compare}
from the proof of Theorem \ref{T:Hoelder} to obtain
\begin{align*}
m (C'''\epsilon)^{n-1} |\bar x_0 - \bar x_1|^{2n-1}
&=m \delta^{n-1} |\bar x_0 - \bar x_1| 
\\&\lesssim \rho[B_\epsilon(x)]
\\&= o(\epsilon^{n-1}) \qquad {\rm as}\ \epsilon \to 0,
\nonumber
\end{align*}
where the last equality is given by \eqref{p=n condition}.
The limit $\epsilon \to 0$ forces $\bar x_0=\bar x_1$ --- a contradiction!
We therefore conclude $u$ is differentiable,  which implies
$u \in C^1(\Sphere^n)$ by the closedness of $\partial u$ noted after (\ref{subdifferential}).
The continuity of $F=\cexp \circ Du$
follows immediately.
\end{proof}


\begin{corollary}\label{C:UAC}{\bf (Uniformly Away from Singular set \cite{DelanoeLoeper06})}
Assume the hypotheses of Theorem~\ref{T:C 0 sphere}, namely $(M=\Sphere^n,\dist,\vol)$ is the round
sphere and $F:\Sphere^n \longrightarrow \Sphere^n$ is an optimal map between a probability
measure $\rho$ satisfying (\ref{p=n condition}) and one satisfying
$d \bar \rho / d\vol \ge m>0$
for the (a) Riemannian distance squared or (b) negative logarithm of the Euclidean cost.
Let $N \subset \Sphere^n \times \Sphere^n$ be the set where $c$ is smooth.
There exists a constant $\delta = \delta(\rho,\bar \rho)>0$
such that $\rho$-a.e.\ point $x \in \Sphere^n$ satisfies
\begin{align}\label{N delta}
\dist(F(x),\Sphere^n \setminus \bar N(x)) \ge \delta.
\end{align}
\end{corollary}

\begin{proof}
Since the map $F$ is unique (up to $\rho$-negligible sets) \cite{McCann01},
Theorem \ref{T:C 0 sphere} asserts it costs no generality to take
$F=\cexp \circ Du$ continuous on $\Sphere^n$ and $u^{c^*c} = u \in C^1(\Sphere^n)$.
Recall for cost (a) $\bar N(x) = \{\bar x \in \Sphere^n \mid \dist(x,\bar x) < \pi = \diam_{\Sphere^n} \Sphere^n \}$,
while for cost (b) $\bar N(x) = \{\bar x \in \Sphere^n \mid \bar x \ne x \}$.
In case (b) the desired estimate follows directly from the Lipschitz bound
on $Du$ provided by Proposition \ref{P:semi-convex reflector antenna},
since $|Dc(x,\bar x)| \to \infty$ as $\bar x \to x$.
If the corollary fails in case (a),  there is a sequence of points $x_k \in \Sphere^n$
with $\dist(F(x_k),\Sphere^n \setminus \bar N(x_k)) \to 0$ as $k \to \infty$.
Extracting a subsequential limit yields $x_k \to x \in \Sphere^n$ with
$F(x) \not \in \bar N(x) = \ran \cexp_x$ antipodal to $x$.  Due to the symmetry of the sphere,
the convex set $\partial u(x)$ cannot then consist of a single point of non-zero magnitude;
instead $\partial u(x)= \{p^* \in T^*_x \Sphere^n \mid |p^*| \le \pi \}$ contradicting the
claim $u \in C^1(\Sphere^n)$.
\end{proof}



Finally, we collect the ingredients above to conclude our direct proof of
(a) Loeper's result \cite{Loeper07p},  and (b) his analogous improvements of
Caffarelli, Guti\'errez \& Huang's result for the reflector
antenna \cite{CaffarelliGutierrezHuang}.

\begin{theorem}{\bf (H\"older continuity of optimal maps on the sphere
\cite{Loeper07p}, \cite{CaffarelliGutierrezHuang})}
\label{T:regular sphere}
As in Theorem~\ref{T:C 0 sphere}, assume $\rho$ and $\bar \rho$
are probability measures on the round sphere $(M=\Sphere^n,\dist,\vol)$
satisfying \eqref{p-condition} for some $p \in ]n,\infty]$
and $d\bar \rho(x) / d\vol \ge m>0$ for all $x\in \Sphere^n$.
For either cost
(a) $c(x,\bar x) = \dist^2(x,\bar x)/2$ or (b) $c(x,\bar x) = -\log|x-\bar x|$
the optimal map $F$ between $\rho$ and $\bar \rho$ lies in
$C^{1/5}\cup C^{(1-\frac{n}{p})/(4n-1-\frac{n}{p})}(\Sphere^n, \Sphere^n)$.
\end{theorem}

\begin{proof}
We first set up a domain $N$ to which we shall apply the proof of Theorem~\ref{T:Hoelder}.
Let $g_{ij}$ denote the standard metric tensor on the round sphere $\Sphere^n$.
For each $z \in \Sphere^n$, let $-z$ denote its antipodal point.
Observe that each $g$-geodesic ball $B_R (z)$
appears convex from $z \in \Sphere^n$ for the Riemannian distance squared cost (a),
and from $-z$ for the negative logarithmic cost (b); see
Definition \ref{D:c-convexity}.  In both cases,  the convexity is $2$-uniform
in the sense that $Dc(z, B_R(\pm z))$ is a $2$-uniform subset of $T_z \Sphere^n$.

Let $z$ be an arbitrary point in $\Sphere^n$ and
$\delta = \delta(\rho,\bar \rho)$ as in \eqref{N delta}.
Fix $\pi -\delta < R < \pi$. By compactness, we can choose $0< r < \delta + R -\pi$
sufficiently small that the set $N= B_r (z) \times B_R(\pm z)$
remains vertically-convex,  choosing the plus sign for cost (a) and the minus sign for cost (b).
If we use $g$-geodesic polar coordinates we can regard the geodesic balls as Euclidean balls
in the tangent spaces $T_z \Sphere^n$ and $T_{\pm z} \Sphere^n$.

We now apply the volume comparison argument of the proof of Theorem~\ref{T:Hoelder}
to the domain $N$. A few points should be clarified.
Horizontal convexity of $N$ was used in Theorem~\ref{T:Hoelder} only to be able to
apply Theorem~\ref{T:supporting c-convex}.  But in the present context,  we need not
apply Theorem~\ref{T:supporting c-convex} since we have its conclusion
for $u$ on the whole $\Sphere^n$ directly by combining Theorem~\ref{T:C 0 sphere} with
Lemma \ref{L:supporting c-convex}.  Also $F$ satisfies
$\{ (x, F(x)) \mid x \in B_r (z)\} \subset N$ by Corollary~\ref{C:UAC}.
The remainder of the proof of Theorem~\ref{T:Hoelder} goes through in the present case, yielding
$F \in C^{1/5}_\loc \cup C_{\mathop{\rm loc}}^{(1-\frac{n}{p})/(4n-1-\frac{n}{p})}(B_r(z),\Sphere^n)$.
This local estimate becomes global by compactness of $\Sphere^n$, so the proof is complete.
\end{proof}


\bibliography{../newbib}
\bibliographystyle{plain} 
\end{document}